\documentclass[12pt]{amsart}
\usepackage{amsmath,dsfont,amsfonts,amssymb,amsxtra,latexsym,amscd,enumerate,amsthm,xcolor}
\usepackage{fullpage} 

   \def\to{\rightarrow}
 \def\ls{\lesssim} \def\gs{\gtrsim}
\def\la{\langle} \def\ra{\rangle}

\def\R{\mathbb{R}} \def\C{\mathbb{C}} \def\Z{\mathbb{Z}}
\def\N{\mathbb{N}} \def\S{\mathbb{S}}

\newcommand{\rl}{r_0}

\newcommand{\eps}{\epsilon} \DeclareMathOperator{\dist}{d}
\DeclareMathOperator{\med}{med}

\def\beq{\begin{equation}} \def\eeq{\end{equation}}
 
\def\beq{\begin{equation}} \def\eeq{\end{equation}}

\newtheorem{thm}{Theorem} 
\newtheorem{pro}[thm]{Proposition} \newtheorem{lem}[thm]{Lemma}

\theoremstyle{remark} \newtheorem{rmk}[thm]{Remark}
\theoremstyle{definition} 

\numberwithin{equation}{section} \numberwithin{thm}{section}

\begin{document}
\title[GWP and scattering for the Dirac-Klein-Gordon system]{Global
  well-posedness and scattering for the massive Dirac-Klein-Gordon system in two dimensions}

\author[I.~Bejenaru]{Ioan Bejenaru} \address[I.~Bejenaru]{Department
  of Mathematics, University of California, San Diego, La Jolla, CA
  92093-0112 USA} \email{ibejenaru@math.ucsd.edu}

\author[V.~Borges]{Vitor Borges} \address[V.~Borges]{Department
  of Mathematics, University of California, San Diego, La Jolla, CA
  92093-0112 USA} \email{vborges@ucsd.edu}

\begin{abstract}
 We prove global well-posedness and scattering for the massive
  Dirac-Klein-Gordon system with small and low regularity initial data in dimension two. To achieve this, we impose a non-resonance condition on the masses.

\end{abstract}

\maketitle

\section{Introduction}\label{sect:intro}

The Dirac-Klein-Gordon (DKG) system is a basic model of proton-proton
interactions (one proton is scattered in a meson field produced by a
second proton) or neutron-neutron interaction, see Bjorken and Drell
\cite{BjDr}.  In physics these are known as the strong interactions
which are responsible for the forces which bind nuclei.
  
The mathematical formulation of the DKG system in two dimensions is as follows:
  \begin{equation} \label{DKG} \left\{
    \begin{aligned} 
      & (-i \gamma^\mu \partial_\mu + M ) \psi= \phi \psi   \\
      & (\Box + m^2) \phi = \psi^\dag \gamma^0 \psi.
    \end{aligned}
  \right.
\end{equation}
Here, $\Box$ denotes the d'Alembertian $\Box=\partial_t^2-\Delta_x$,
$\psi: \R^{1+2} \rightarrow \C^2$ is the spinor field (column vector),
and $\phi: \R^{1+2}\rightarrow \R$ is a scalar field. For
$\mu=0,1,2$, $\gamma^\mu$ are the $2 \times 2$ Dirac matrices
\[
\gamma^0= \left( \begin{array}{cc} 1 & 0 \\ 0 & -1 \end{array}
\right) , \qquad \gamma^j=\gamma^0 \sigma^j,
\]
where $\sigma^j, j=1,2,3$, are the Pauli matrices given by
\[
\sigma^1= \left( \begin{array}{cc} 0 & 1 \\ 1 & 0 \end{array} \right)
, \qquad \sigma^2=\left( \begin{array}{cc} 0 & -i \\ i & 0 \end{array}
\right) , \qquad \sigma^3=\left( \begin{array}{cc} 1 & 0 \\ 0 &
    -1 \end{array} \right);
\]
we included the full set of Pauli matrices, although $\sigma^3$ is not used in the equation.
$\psi^\dag$ denotes the conjugate transpose of $\psi$, i.e.\
$\psi^\dag=\overline{\psi}^t$. The matrices $\gamma^\mu$ satisfy the
following properties
\[
\gamma^\alpha \gamma^\beta + \gamma^\beta \gamma^\alpha = 2 g^{\alpha
  \beta} I_2, \qquad g^{\alpha \beta}= \mathrm{diag}(1,-1,-1).
\]
We study the Cauchy problem with initial condition
\begin{equation}\label{eq:i-cond}(\psi,\phi,\partial_t
  \phi)|_{t=0}=(\psi_0, \phi_0, \phi_1).
\end{equation}
Based on physical considerations found in Bjorken and Drell \cite[Chapter 10.2]{BjDr}, the first author and Herr explained in \cite{BH-DKG} that it is reasonable to assume that in \eqref{DKG} the following holds true
\[
2M > m > 0.
\] 
In \cite{BH-DKG} it was highlighted that this translates into a non-resonance condition for this system; more details are provided in Section \ref{subsect:mod}.

We now turn our attention to the mathematical aspects of \eqref{DKG}. The fundamental question is that of global regularity of solutions. Most of the methods were developed originally for the three dimensional version of \eqref{DKG}, see for instance \cite{ChGl, Kl, Ba, BeBe, Bo, DaFoSe, Wa, BH-DKG}. In most of these results the masses $M,m$ are arbitrary; only in \cite{Wa, BH-DKG} does the non-resonant condition $2M>m>0$ appear. 

We now explain the current state of the DKG system in two dimensions. For this
we consider the Cauchy problem for \eqref{DKG} with initial data
\begin{equation} \label{rs}
\psi_0 \in H^s(\R^2) ,\quad (\phi_0,\phi_1) \in H^{r}(\R^2) \times
H^{r-1}(\R^2).
\end{equation}
In \cite{Bo2D} Bournaveas  established local well-posedness for $s > \frac14, r=s+\frac12$. This was improved by D'Ancona, Foschi and Selberg \cite{DaFoSe-2D} to the range $s > -\frac15$ and $\max(\frac14-\frac{s}2, \frac14 + \frac{s}2, s) < r < \min(\frac34+2s, \frac34 + \frac{3s}2, 1+s)$ by taking advantage of the null structure which they found in earlier work on the three dimensional problem, see \cite{DaFoSe}. There is also the result of Selberg and Tesfahun in \cite{SeTe} which establishes unconditional uniqueness for  $s \geq 0, r=s+\frac12$. 

In \cite{GrPe}, Gr\"unrock and Pecher establish the first global well-posedness for \eqref{DKG} in two dimension. Their approach relies on combining a refined version of the local in time theory with the conservation law $\|\psi(t)\|_{L^2}=constant$ to obtain a global in time theory; this draws ideas from previous work of Colliander, Holmer and Tzirakis on the Zakharov system in \cite{CoHoNi} and a follow-up work of the first author, Herr, Holmer and Tataru in \cite{BHHT}. The precise results states that \eqref{DKG} with data $\psi(0) \in L^2$ and $(\phi(0), \partial_t \phi(0)) \in H^\frac12 \times H^{-\frac12}$ is globally well-posed; moreover persistence of higher regularity is established. It is important to highlight that this result has no restriction on the size of the initial data. One missing piece of information in the theory developed in \cite{GrPe} is the scattering; this is a normal feature when a global theory relies on the equation being energy subcritical. 

In the context of initial data with high regularity and spatial localization (or concentration),
linear scattering for \eqref{DKG} has been established in \cite{DoLiMaYu, Zh, DoWy}. 

In this paper we seek to develop a global theory for \eqref{DKG} for initial data with limited regularity and no spatial decay, and without relying on energy conservation. We rely on a global in time iteration scheme in carefully designed resolution spaces, and, as a byproduct, we do obtain linear scattering.  Our main result is the following

\begin{thm}\label{thm:main}
  Assume that $\epsilon > 0$ and $2M > m >0$. Then the Cauchy problem
  \eqref{DKG}-\eqref{eq:i-cond} is globally well-posed for small
  initial data
  \[\psi_0 \in H^{\frac12+\epsilon}(\R^2;\C^4),\; (\phi_0,\phi_1) \in
  H^{1+\epsilon}(\R^2;\R) \times H^{\epsilon}(\R^2;\R)\] and
  these solutions scatter to free solutions for $t\to \pm \infty$.
\end{thm}
We refer to Subsection \ref{subsect:proof} for more details.
Our result requires more regularity than the ones in \cite{Bo2D, DaFoSe-2D, GrPe}; on the other hand, it provides a global iteration scheme and linear scattering, while the iteration schemes in these papers work only locally in time. Compared with more recent work in \cite{DoLiMaYu, Zh, DoWy} where linear scattering is established, our result above assumes no spatial localization of the data and significantly less regularity.

We note that the scale invariant data space for \eqref{DKG} corresponds to $s=-\frac12,r=0$. We have no reason to believe that the regularity threshold in Theorem \ref{thm:main} is optimal, and we recall that local theories are available below this threshold; see \eqref{rs} and the following paragraphs. We will investigate the possibility of lowering the regularity threshold in Theorem \ref{thm:main}.

As suggested by the scaling of the massless equation, we work with $s=r+\frac12$ in \eqref{rs}. It is likely that our analysis allows for more general relations between $r$ and $s$, but we do not pursue this here.

We highlight that in the results in \cite{Bo2D, DaFoSe-2D, GrPe} the masses
$M,m$ are arbitrary. In the context of a local-in-time result, the terms $M \psi$, $m^2 \phi$ can be treated as perturbations, thus allowing an analysis of \eqref{DKG}
as a system of wave equations. However, this cannot be the case for a
global in time theory that includes scattering.

Our assumption $2M > m >0$ implies that the system
\eqref{DKG} has no resonances. It was known from prior works on
Klein-Gordon type systems with multiple speeds that, under certain
conditions between the masses, resonant interactions do not occur and
the well-posedness theory improves. We refer the reader to the works
of Delort and Fang \cite{DeFa}, Schottdorf \cite{S12} and Germain
\cite{Ge} and to the references therein. 

Note that unlike many of the previous works that dealt with power-type nonlinearities for the Klein-Gordon equation, the
DKG system contains derivatives. This is not apparent
from our formulation of \eqref{DKG}; however, if one wants to write
\eqref{DKG} as a system of Klein-Gordon equations, one should apply
$(-i \gamma^\mu \partial_\mu - M )$ to the first equation and then it
is obvious that the right-hand side contains derivatives.

We now explain the main ideas in this paper. The starting point is to implement a similar strategy to the one developed by the first author and Herr in \cite{BH-DKG}
for the three dimensional problem. This involves sorting out the interplay between the null structure and the lack of resonances, and building the resolution space based on localized Strichartz estimates and $X^{s,\frac12,\infty}$ spaces; practically this means adapting the setup from \cite{BH-DKG} to two dimensions. An iteration scheme based only on this type of information fails due to a loss of derivative (in high frequency) in the nonlinear estimates; this cannot be fixed by assuming more regularity on the initial data and hence for the solution. Our solution is to involve a more general version of the $X^{s,\frac12,\infty}$ type spaces; the generality comes from replacing the $L^2_{t,x}$ base space in the definition of $X^{s,\frac12,\infty}$ by $L^p_t L^q_x$, see section \ref{subsec:ressp} for more details. These structures are particularly useful when $p< 2$ since they give access to more time integrability (which we use for low frequencies), and, in nonlinear interactions, one can use less time integrability on other factors (which we do for high frequencies). Since the standard Strichartz estimates for Klein-Gordon requires more regularity as one increases the time integrability, this mechanism allows us to fix the loss of derivative in high frequency mention earlier. 

A natural question to ask is what is the mechanism that allows us to introduce these structures. The short answer is that this mechanism is already present in the simpler, more convenient setup of the $X^{s,\frac12,\infty}$ space (or variations of). The $X^{s,\frac12,\infty}$ space provides information that is not available at the linear level; the basic idea is that outputs at higher modulation are created via nonlinear interactions. The output of these nonlinear interactions can be (potentially) measured in $L^2_{t,x}$ spaces, but also in many other spaces that appear from considering products of, say, linear waves that are estimated in Strichartz norms. 
One of the main reasons one chooses to take advantage of this feature using  $X^{s,\frac12,\infty}$ (or $X^{s,\frac12,r}, r \in [1,\infty]$) is that it is $L^2_{t,x}$ based and one can rely on Plancherel identity to deal with it at the technical level.
Involving any other $L^p_t L^q_x$-based structure (other than $L^2_t L^2_x$) requires additional work precisely due to the lack of the Plancherel identity. Our paper is not the first time that non-$L^2_{t,x}$ based structures are being used, see for instance \cite{BH-CD2, BH-CD3} (even here one uses $L^p_t L^2_x$ based structures which still carry the benefit of having a partial $L^2$ based structure). What sets our approach apart is the central role they play in closing the iteration scheme.

An alternative, yet natural, way to take advantage of the lack of resonance would be a normal form approach that increases the order of nonlinearity (from quadratic to cubic in our case); see also the Acknowledgments at the end of this section. The original work of Shatah in \cite{Sh} on the Klein-Gordon equation with quadratic nonlinearities introduced the normal form transformation to bypass problems similar to the ones mentioned above; note that our system \eqref{DKG} is in fact reducible to a system of Klein-Gordon type equations with quadratic nonlinearities in dimension two, just like in \cite{Sh}. The approach introduced in \cite{Sh} has been used in many other works, see for instance \cite{Ge, OzTsTs, GuNa} to name just a few instances that are closer to our context. We believe that our approach in this paper provides an alternative to the normal form transformation, by taking advantage of a wider class of estimates that are available in higher modulation which essentially translates into increasing the order of the nonlinearity when needed, just like the normal form method does. We also think that this is one of the strengths of this work, which may be of independent interest and applicable to other problems.

We conclude this section with an overview of the paper. In Section \ref{sect:red}
we introduce some of the basic notation and rewrite the original system \eqref{DKG} in 
the equivalent form \eqref{DKGf} which has two advantages: it is first-order in time and it
unveils the null structure. The gains from the null structure are quantified in Subsection \ref{subsect:mod}
in a manner that fits our analysis. In Section \ref{sect:fs-linear} we introduce the resolution space in which we iterate our system and establish the basic linear estimates. This is where we formalize the new structures generalizing the $X^{s,\frac12,\infty}$ space as described earlier. 
Just as in \cite{BH-DKG} we employ the $U^p,V^p$ spaces in order to establish some of the basic linear estimates. In Section \ref{sect:nl} we prove some key trilinear estimates which are then used to prove our main result in Theorem \ref{thm:main}.

\bf Acknowledgments: \rm The first author is grateful to Sebastian Herr and Achenef Tesfahun for many useful conversation on this problem. In particular the first author, Herr and Tesfahun explored the possibility of employing the normal form method for this problem. 

The first author was supported by the NSF grant DMS-1900603.

\section{Notation and Reductions}\label{sect:red}

The content of this section is nearly identical to that in the corresponding section in \cite{BH-DKG}. 

\subsection{Notation}\label{subsect:not}
We define $A\ls B$, if there is a harmless constant $c>0$ such that $A\leq c B$, and $A\gs B$ iff $B\ls A$. Further, we define $A\approx B$ iff both $A\ls B$ and $B\ls A$. Also, we define $A\ll B$ if the constant $c$ can be chosen such that $c<2^{-10}$. Also, $A\gg B$ iff $B\ll A$.

Similarly, we define $A\preceq B$ iff $2^A\ls 2^B$, $A\succeq B$ iff $2^A\gs 2^B$, $A \sim B$ iff $2^A\approx 2^B$, $A\prec B$ iff $2^A\ll 2^B$, $A\succ B$ iff $2^A\gg 2^B$.

Let $\rho^0\in C^\infty_c(-2,2)$ be a fixed smooth, even, cutoff
satisfying $\rho^0(s)=1$ for $|s|\leq 1$ and $0\leq \rho\leq 1$. For
$k \in \Z$ we define $\rho_k:\R \ \mbox{or} \ \R^2 \to \R$,
$\rho_k(y):=\rho^0(2^{-k}|y|)-\rho^0(2^{-k+1}|y|)$; it will be clear from the context if the argument of $\rho_k$ is in $\R$ or $\R^2$. Let $\tilde{\rho}_k=\rho_{k-1}+\rho_k+\rho_{k+1}$. For $k \geq 1$, let $P_k$ be the
Fourier multiplication operators with respect to $\rho_k$, and
$P_0=I-\sum_{k \geq 1}P_k$.  For $j \in \Z$ we define
\[
\mathcal{F}[Q^{\pm,m}_{j}f](\tau,\xi)=\rho_j(\tau\pm \la\xi
\ra_m)\mathcal{F}f(\tau,\xi).
\]
Similarly, we define $\tilde{P}_k$ and $\tilde{Q}^{\pm,m}_{j}$. Next we define $P_{\leq k}=\sum_{0\leq k'\leq k }P_{k'}$, $P_{\prec
  k}=\sum_{0\leq k'\prec k }P_{k'}$, $P_{> k}=I-P_{\leq k}$,
$P_{\succeq k}=I-P_{\prec k}$; we define $Q^{\pm,m}_{\leq j}=\sum_{ j'\leq j } Q^{\pm,m}_{j'}$, and similarly $Q^{\pm,m}_{\prec j}$, $Q^{\pm,m}_{\succeq j}$, and $Q^{\pm,m}_{j \in J}$ for an interval $J$.
In the obvious way we also define the analogous operators based on $\tilde{P}_k$ and $\tilde{Q}^{\pm,m}_{j}$.

In the case $m=1$ we suppress the superscripts, e.g.\
$Q^{\pm,1}_{j}=Q^{\pm}_{j}$.

Further, for $l \in \N$ let $\mathcal{K}_{l}$ denote a set of
spherical caps of radius $2^{-l}$ which is a covering of $\S^1$ with
finite overlap. For a cap $\kappa \in \mathcal{K}_{l}$ we denote its
center in $\S^1$ by $\omega(\kappa)$. Let $\Gamma_\kappa$ be the cone
generated by $\kappa \in \mathcal{K}_{l}$ and $(\eta_\kappa)_{\kappa
  \in \mathcal{K}_l}$ be a smooth partition of unity subordinate to
$(\Gamma_\kappa)_{\mathcal{K}_{l}}$. Let $P_{\kappa}$ denote the
Fourier-muliplication operator with symbol $\eta_\kappa$, such that
$I=\sum_{\kappa \in \mathcal{K}_l}P_{\kappa}$. Further, let
$\tilde{P}_{\kappa}$ with doubled support such that
$P_{\kappa}=\tilde{P}_{\kappa}P_{\kappa}=P_{\kappa}\tilde{P}_{\kappa}$.
For notational convenience, we also define $\mathcal{K}_{0}=\{\S^1\}$
and $P_\kappa=I$ if $\kappa \in \mathcal{K}_{0}$.

\subsection{Setup of the system and null
  structure}\label{subsect:setup-null}

As written in \eqref{DKG} the DKG system has a
linear part whose coefficients are matrices; however it is technically
easier to work with scalar equations. Such a reduction was developed by D'Ancona, Foschi and Selberg in  \cite[Section 2 and 3]{DaFoSe} in 
three dimensions and in \cite{DaFoSe-2D} in two dimensions. This setup is also effective in identifying a null-structure in the nonlinearity. 

In  \cite{BH-DKG} the first author and Herr adapted (in three dimensions) the framework from \cite[Section 2 and 3]{DaFoSe} to include the mass terms in the linear part, which is very useful if one seeks to develop a global theory. Below we use this setup and note that the adjustments to our two dimensional setup are very minor and mainly at the notation level. 

For $j=1,2$ the matrices
$
\alpha^j:=\gamma^0\gamma^j, \; \beta:=\gamma^0
$
have the properties
\[
\alpha^j\beta+\beta\alpha^j=0, \;\alpha^j \alpha^k+\alpha^k\alpha^j=2\delta^{jk}I_2,
\]
see \cite[p.~878]{DaFoSe} for more details. We introduce the Fourier multipliers $\Pi_{\pm}^M(D)$ with symbol
\[
\Pi^M_\pm(\xi)=\frac12 [I \pm \frac{1}{\la \xi \ra_M} ( \xi \cdot
\alpha + M \beta)].
\]
In the case $M=1$ we suppress the superscript, i.e.\
$\Pi_{\pm}(D)=\Pi_{\pm}^1(D)$.

We then define $\psi_\pm=\Pi_\pm^M(D) \psi$ and split $\psi=\psi_+ +
\psi_-$. Also, define $\la D \ra_M=\sqrt{M^2-\Delta}$, with the usual convention that $\la D \ra = \la D \ra_1$. 
By applying the operators $\Pi_\pm^M(D)$ to the system
\eqref{DKG} we obtain the following system of equations:
\begin{equation} \label{DKGnew} \left\{
    \begin{aligned} 
      & (-i\partial_t + \langle D \rangle_M) \psi_+ = \Pi^M_+(D) (\phi \beta \psi) \\
      & (-i\partial_t - \langle D \rangle_M) \psi_- = \Pi^M_-(D) ( \phi \beta \psi)  \\
      & (\Box + m^2) \phi = \la \psi, \beta \psi \ra.
    \end{aligned}
  \right.
\end{equation}
In order to have a fully first order system, we define
$\phi_\pm=\phi\pm i \la D \ra_m^{-1} \partial_t \phi$ thus
\[
(-i \partial_t + \la D \ra_m) \phi_+ = \la D \ra_m^{-1} \la \psi,
\beta \psi \ra.
\]
Note that $\phi=\Re \phi_\pm$ and $\phi_-=\overline{\phi_+}$ since
$\phi$ is real-valued.  The system which we will study is
\begin{equation} \label{DKGf} \left\{
    \begin{aligned} 
      & (-i\partial_t + \langle D \rangle_M) \psi_+ = \Pi^M_+(D) (\Re \phi_+ \beta \psi) \\
      & (-i\partial_t - \langle D \rangle_M) \psi_- = \Pi^M_-(D) ( \Re \phi_+ \beta \psi)  \\
      & (-i \partial_t +\la D \ra_m) \phi_+ = \la D \ra_m^{-1} \la
      \psi, \beta \psi \ra.
    \end{aligned}
  \right.
\end{equation}
We aim to provide a global theory for this system for initial data
$(\psi_{\pm,0}, \phi_{+,0}) \in H^{\frac12+\epsilon} \times H^{1 +
  \epsilon}$.  It is an easy exercise that this translates back into a
global theory for the original system with $(\psi_0,\phi_0,\phi_1) \in
H^{\frac12+\epsilon} \times H^{1 + \epsilon} \times H^{
  \epsilon}$.

There is a null structure in the system \eqref{DKGf}, which we
describe next.  This is again inspired by the work in \cite{DaFoSe}
and was adapted to the current setup in \cite{BH-CD3, BH-CD2, BH-DKG}. For more details,
we refer to the reader to \cite{DaFoSe,BH-CD3}. We decompose $\la \psi, \beta \psi \ra$ as
\[
\begin{split}
  \la \psi, \beta \psi\ra & = \la \Pi^M_+(D) \psi_+, \beta \Pi^M_+(D) \psi_+ \ra + \la \Pi^M_-(D) \psi_-, \beta \Pi^M_-(D)) \psi_- \ra \\
  &\quad + \la \Pi^M_+(D) \psi_+, \beta \Pi^M_-(D) \psi_- \ra+ \la
  \Pi^M_-(D) \psi_-, \beta \Pi^M_+(D) \psi_+ \ra.
\end{split}
\]
We have
\begin{equation}\label{eq:com}
  \Pi^M_\pm(D)\beta=\beta\Pi^M_\mp(D)\pm M \la D\ra^{-1}_M \beta. 
\end{equation}

The following Lemma, which corresponds to \cite[Lemma 3.1]{BH-CD3},  \cite[Lemma 3.1]{BH-CD2} and
\cite[Lemma 2]{DaFoSe}, analyzes the symbols of the bilinear operators
above.

\begin{lem}\label{lem:PiPi} For fixed $M \geq 0$, the following holds true:
  \begin{equation} \label{PiPi}
    \begin{split}
      \Pi_\pm^M(\xi) \Pi_\mp^M(\eta) & = \mathcal{O}(\angle (\xi,\eta)) + \mathcal{O}(\la \xi \ra^{-1} + \la \eta \ra^{-1}) \\
      \Pi_\pm^M(\xi) \Pi_\pm^M(\eta) & = \mathcal{O}(\angle
      (-\xi,\eta)) + \mathcal{O}(\la \xi \ra^{-1} + \la \eta \ra^{-1}).
    \end{split}
  \end{equation}
\end{lem}

We now explain heuristically why this is useful here, see Lemma
\ref{lem:stable} for the technical result which will be used in the
nonlinear analysis. By \eqref{eq:com} it follows that for $s_1,s_2 \in \{+,-\}$
\begin{align*}
  & \mathcal{F}_x\la \Pi_{s_1} \psi_1,\beta\Pi_{s_2} \psi_2\ra(\xi)=\int\limits_{\xi=\xi_1-\xi_2}\la \Pi_{s_1}(\xi_1) \widehat{\psi_1}(\xi_1),\beta\Pi_{s_2}(\xi_2) \widehat{\psi_2}(\xi_2)\ra d\xi_1d\xi_2\\
  =& \int\limits_{\xi=\xi_1-\xi_2} \left( \la \beta\Pi_{-s_2}(\xi_2)\Pi_{s_1}(\xi_1) \widehat{\psi_1}(\xi_1),\widehat{\psi_2}(\xi_2)\ra 
  +s_2 M
  \la\xi_1\ra_M^{-1}\la\beta\Pi_{s_1}(\xi_1) \widehat{\psi_1}(\xi_1),\widehat{\psi_2}(\xi_2)\ra \right)  d\xi_1d\xi_2.
\end{align*}
Hence, smallness of the angle $\angle(s_1\xi_1,s_2\xi_2)$ can be
exploited as long as it exceeds
$\max(\la\xi_1\ra_M^{-1},\la\xi_2\ra_M^{-1})$.  See
\cite[p.~885]{DaFoSe} for the analogue of this in the massless case,
where we have $\Pi_-^0(\xi_1) \Pi_+^0(\xi_2) =0$ if $\angle
(\xi_1,\xi_2)=0$, which makes the null structure effective at all
angular scales. In the massive case $M>0$ the null-structure does not
bring gains beyond $\max(\la\xi_1\ra_M^{-1},\la\xi_2\ra_M^{-1})$.  To
compensate for this we need to use that there are no resonances
present in \eqref{DKGf}.

\subsection{Modulation analysis}\label{subsect:mod}

As detailed in \cite{BH-DKG}, the nonlinearity in the DKG system exhibits two key features:  
lack of resonant interactions and a correlation between smallness of the maximal modulation and angular constraints.

There are many instances in the literature where the lack of resonances has been exploited, see for instance  
 \cite{S12,DeFa,Ge}. In our context, the first author and Herr had exploited this phenomena for the cubic Dirac in \cite{BH-CD3,BH-CD2}
and for the DKG system in three dimensions in \cite{BH-DKG}. Below we recall the two results from \cite{BH-DKG} which quantify the phenomenas mentioned above; the corresponding claims in  \cite{BH-DKG} are in three dimensions and they imply the two dimensional analogues stated below.

The first result is  \cite[Lemma 2.2]{BH-DKG} and provides lower bounds for the resonance function.

\begin{lem}\label{lem:res}
  Fix $0<m<2M$. For $s_1,s_2\in \{+,-\}$ define the resonance function
  \begin{equation}\label{eq:res-fct}
    \mu^{s_1,s_2}(\xi_1,\xi_2):=\la\xi_1-\xi_2\ra_m+s_1\la \xi_1\ra_M-s_2\la \xi_2\ra_M.
  \end{equation}
  Then, we have the following bounds:

  {\it Case 1:} If
  \begin{enumerate} \item[a)]$s_1=+, s_2=-$ or
  \item[b)] $s_1=-,s_2=+$ and $\la\xi_1-\xi_2\ra_m\ll \min(\la
    \xi_1\ra_M, \la \xi_2\ra_M)$,
  \end{enumerate}
  then
  \begin{equation}\label{eq:high-mod}
    |\mu^{s_1,s_2}(\xi_1,\xi_2)|\gs \max(\la \xi_1-\xi_2\ra,\la \xi_1\ra,\la \xi_2\ra)
  \end{equation}

  {\it Case 2:} If \begin{enumerate} \item[a)] $s_1=s_2$ or
  \item[b)] $s_1=-,s_2=+$ and $\la\xi_1-\xi_2\ra_m\gs \min(\la
    \xi_1\ra_M, \la \xi_2\ra_M)$,
  \end{enumerate}
  then
  \begin{equation}\label{eq:mod-angle}
    \begin{split}
      |\mu^{s_1,s_2}(\xi_1,\xi_2)|\gs_{m,M} &\frac{\la \xi_1\ra \cdot
        \la \xi_2\ra}{\la \xi_1-\xi_2\ra}\angle(s_1\xi_1,s_2\xi_2)^2
    \end{split}
  \end{equation}
  With any choice of signs, we have both
  \begin{equation}\label{eq:gen-lb}
    |\mu^{s_1,s_2}(\xi_1,\xi_2)|\gs_{m,M}\min(\la \xi_1\ra , \la \xi_2\ra )\angle(s_1\xi_1,s_2\xi_2)^2,
  \end{equation}
  and the \emph{non-resonance} bound
  \begin{equation}\label{eq:non-res}
    |\mu^{s_1,s_2}(\xi_1,\xi_2)|\gs_{m,M} \max(\la \xi_1-\xi_2\ra^{-1},\la \xi_1\ra^{-1},\la \xi_2\ra^{-1}).
  \end{equation}
\end{lem}

\begin{rmk}\label{rmk:m}
From now on, we fix $M=m=1$ in order to simplify the exposition. In view
of Lemma \ref{lem:res}, it will be obvious that all arguments carry
over to the case $2M>m>0$ with modified (implicit) constants depending
on $m,M$.
\end{rmk}

\begin{lem} \label{lem:mod} Let $s_1,s_2\in \{+,-\}$. Consider
  $k,k_1,k_2\in \N_0, j,j_1,j_2 \in \Z$, and $\phi=\tilde{P}_{k}
  \tilde{Q}^{+}_{j}\phi$, $u_i=\tilde{P}_{k_i}
  \tilde{Q}^{s_i}_{j_i}u_i$.

  i) If $\max(j,j_1,j_2) \prec -\min(k,k_1,k_2)$, we have
  \begin{equation}\label{eq:int-zero}
    \int_{\R^{1+2}}\phi \cdot u_1\overline{u_2} \, dtdx =0.
  \end{equation}

  ii) {\it Case 1:} Suppose that \begin{align*}& s_1=+,s_2=-\\
\text{ or }\qquad & s_1=-,s_2=+ \text{ and }k\prec \min(k_1,k_2).\end{align*}
If $\max(j,j_1,j_2)\prec \max(k,k_1,k_2)$,
  then, \eqref{eq:int-zero} holds true.

  {\it Case 2:} Suppose that \begin{align*}&s_1=s_2\\
\text{ or }\qquad & s_1=-,s_2=+ \text{ and }k\succeq
  \min(k_1,k_2).
\end{align*}
If $l\geq 1$, $\kappa_1,\kappa_2 \in
  \mathcal{K}_{l}$ with $\dist(s_1\kappa_1,s_2\kappa_2)\geq 2^{-l}$
  and $\max(j,j_1,j_2)\prec k_1+k_2-k-2l$, then
  \begin{equation}\label{eq:int-zero-caps}
    \int_{\R^{1+2}}\phi \cdot \tilde{P}_{\kappa_1} u_1\overline{\tilde{P}_{\kappa_2}u_2} \, dtdx =0.
  \end{equation}
iii) Assume that $k\prec \min(k_1,k_2), |k_1-k_2| \leq 10$ and $l= k_1 -k$. 
Then, if $\kappa_1,\kappa_2 \in  \mathcal{K}_{l}$ with $\dist(\kappa_1, \kappa_2) \succ 2^{-l}$, the following holds true
 \begin{equation}\label{eq:int-zero-caps-2}
    \int_{\R^{1+2}}\phi \cdot \tilde{P}_{\kappa_1} u_1\overline{\tilde{P}_{\kappa_2}u_2} \, dtdx =0.
  \end{equation}
\end{lem}
We remark that while the result in iii) is not provided in \cite[Lemma 2.4]{BH-DKG}, its statement appears in the body of the proof the Lemma in that paper. It follows from the simple identity
  \begin{equation}\label{eq:ft}
    \int_{\R^{1+2}}\phi \cdot u_1\overline{u_2} \, dtdx
    =
    \iint \widehat{\phi}(\zeta_2-\zeta_1)\widehat{u_1}(\zeta_1)\overline{\widehat{u_2}}(\zeta_2)d\zeta_1d\zeta_2
  \end{equation}
where $\zeta_j=(\tau_j,\xi_j)$.

\section{Function spaces and linear estimates}\label{sect:fs-linear}

In this section, we develop the theory for the following linear PDE, which is the basic model used in \eqref{DKGf}, 
\[
-i\partial_t u \pm \la D \ra u=f, \quad u(0)=u_0.
\]
First we construct the resolution space which contains enough structure for the later nonlinear analysis, and then we establish the basic linear estimates on $u$ in the resolution space in terms of the initial data $u_0$ and the forcing $f$; the linear estimate is stated in a dual formulation (in terms of information on $f$) which is beneficial in the later nonlinear analysis.

\subsection{Resolution space} \label{subsec:ressp}

Our resolution space contains the following structures:

\begin{itemize}

\item \bf Strichartz structures. \rm These are based on the standard Strichartz estimates available for the $2D$
Klein-Gordon equation. Just as in \cite{BH-DKG} we use localized versions of these estimates. 

\item \bf $\dot{X}^{\pm,\frac12,\infty}$ type spaces. \rm These are the standard $X^{s,b}$ type structures adapted to our context; basically they allow the high modulation part of the function to be measured in $L^2_{t,x}$ type spaces with a weight depending on the modulation.

\item \bf $\dot{X}_{\frac43,4}^{\pm,\frac12,1}$ structure. \rm These are the new structures which we introduce in this paper. Just like the $\dot{X}^{\pm,\frac12,\infty}$ structure mentioned above, their role is to highlight better space-time integrability in high modulation when compared with what the Strichartz estimates provide. These spaces provide more ``time" integrability, $L^\frac43_t$ to be more precise, than the $\dot{X}^{\pm,\frac12,\infty}$ (which provides only $L^2_t$); the price to pay in using these structures is that, since they are not based on $L^2_{t,x}$ building blocks, one cannot rely on Plancherel to have a good representation on the Fourier side and this makes their use/characterization more technical. 

\end{itemize}

We now turn to the details of this construction. For $1\leq p\leq \infty$, $b \in \R$, we define
\[\|f\|_{\dot{X}^{\pm,b,p}}=\big\|\big(2^{bj}\|Q_{j}^\pm
f\|_{L^2}\big)_{j \in \Z }\big\|_{\ell^p},\]

The low frequency part will be treated as a whole, that is we define
\[
\| f \|_{S^\pm_{\leq 0}} = \| f \|_{L^\infty_t L^2_x} + \| f \|_{L^\frac83_t
  L^8_x} + \|f \|_{\dot{X}^{\pm,\frac12,\infty}} + \sum_{j \succeq 0} 2^\frac{j}2 \|Q^\pm_j f \|_{L^\frac43_t L^4_x}.
\]
By interpolation, the space above provides a wide range of Strichartz
estimates for the Sch\"odinger equation on $\R^2$, precisely any $L^p_t L^q_x$ with $\frac1p +\frac1q=\frac12$ and $\frac83 \leq p \leq +\infty$; 
in fact, our theory provides all the known Strichartz estimates for the Sch\"odinger equation on $\R^2$, that is for any $2 < p \leq \infty$, but the ones listed above are sufficient in this paper. This is because, in the low frequency regime,  the Klein-Gordon equation  behaves like the
Sch\"odinger equation.

The third component is the standard modification of the $X^{s,b}$ structure, which measures pieces localized at modulation $\approx 2^j$ in $L^2_{t,x}$ spaces with the appropriate weight. The last structure is novel for this problem. Scaling-wise it sits at the same level as the $L^2_{t,x}$ space used in the  $\dot{X}^{\pm,\frac12,\infty}$ space, but it does bring in more integrability in the time variable which will be key in closing our arguments. 

In high frequency, the Klein-Gordon equation is of wave type and the
Strichartz estimates should reflect that.  Moreover we need some
refinement of the standard Strichartz estimates. We recall from \cite{BH-CD2} the standard Strichartz estimates for the Klein-Gordon equation.

\begin{pro}\label{lem:full-str}
  Let $p,q\geq 2$ such that $(p,q)$ is a Schr\"odinger-admissible
  pair, i.e.\
  \[
  (p,q)\ne (2,\infty), \; \frac{1}{p}+\frac{1}{q}= \frac12,
  \text{and }s=1-\frac{2}{q}.
  \] 
  {\rm i)} Then, for any $k \geq 0$, we have
  \begin{equation}\label{eq:full-str}
    \|P_k e^{\pm it \la D \ra } f \|_{L^p_t(\R;L^q_x(\R^2))} \ls 2^{k s} \|P_k f\|_{L^2}.
  \end{equation}
  {\rm ii)} Moreover,  for any $k \geq 0$, we have
  \begin{equation}\label{eq:full-str-loc}
    \sup_{1\leq l \leq k+10} \Big(\sum_{\kappa \in \mathcal{K}_l} \|P_k P_\kappa e^{\pm it \la D \ra } f \|_{L^p_t(\R;L^q_x(\R^2))}^2 \Big)^{\frac{1}{2}} 
    \ls 2^{ks} \|f\|_{L^2}.
  \end{equation}
\end{pro}

We remark the in part i) above we can freely replace  $P_0$ with  $P_{\leq 0}$; there is no need to consider the equivalent of part ii) for the setup with $P_{\leq 0}$.  

We note that while in \cite{BH-CD2} the statement is made for functions in the more complicated spaces $S_k^\pm$, the proof there begins by establishing the estimates for free solutions (simply based on kernel estimates and a $TT^*$ type argument), which is exactly what we claim above.

Now, we consider functions in $f\in L^\infty_t(\R;L^2(\R^2;\C^d)$ (we use $d=1$ for the Klein-Gordon part and $d=2$ for the Dirac part).
For $k \in \Z, k \geq 0$ and $l, k' \in \Z, 0 \leq l \leq k$, we define
\[
\| f \|_{L^p_t L^q_x[k;l]} : = \left(  \sum_{\kappa \in \mathcal{K}_l}  \| P_{\kappa}  f \|^2_{L^p_t L^q_x} \right)^\frac12.
\]

For $k \geq 0$, we define
\begin{equation}\label{eq:sk}
\begin{split}
  \| f \|_{S_k^\pm}  & = \| f \|_{L^\infty_t L^2_x} + \| f \|_{\dot X^{\pm,\frac12,\infty}} +  \\
  &  +  \sup_{j \in \Z \cup \{+\infty\}} \sup_{0 \leq l \leq k} \left( 2^{-\frac{3k}4} \| Q^\pm_{\leq j} f \|_{L^\frac83_t L^8_x[k;l]} + 2^{-\frac{k}4} \| Q^\pm_{\leq j} f \|_{L^8_t L^\frac83_x[k;l]} \right).
\end{split}
\end{equation}
In the above we use the convention that $Q_{\leq +\infty}=I$, the identity operator. Thus, the $S_k^\pm$ structure gives the standard cap-localized Strichartz estimates in $L^\frac83_t L^8_x[k;l]$ and $L^8_t L^\frac83_x[k;l]$ not only for $f$, but also for $Q^\pm_{\leq j} f$. This is helpful, since if we did not include the operators $Q^\pm_{\leq j}$ in the definition above, then we would have to deal with the stability of $S_k^\pm$ under $ Q^\pm_{\leq j}$; this is possible just as it was done in \cite{BH-DKG}, but in our context we would encounter some unpleasant logarithmic terms whose fixing would require unnecessary effort. We also note that in practice we need only the range $-k \preceq j \leq k$ and the value $j=+\infty$.

In addition, we involve the following high-modulation structure. Given $k \in \Z$ and $j \in \Z, j \succeq -k$ we let $l$ be given by the relation  $j=k-2l$. With this choice we let
\[
\begin{split}
\| Q^\pm_j f\|_{X^{\pm, b}_{k,p,q}} & = 2^{jb}  \left( \sum_{\kappa \in \mathcal{K}_{l}: 2l+j=k} \|Q^\pm_j P_\kappa f \|_{L^p_t L^q_x}^2 \right)^\frac12, \quad - k \preceq j \leq k, \\
\| Q^\pm_j f\|_{X^{\pm, b}_{k,p,q}} & = 2^{jb} \|Q^\pm_j  f \|_{L^p_t L^q_x}, \quad j \geq k. 
\end{split}
\]
Based on this, for $1 \leq r \leq \infty$, we define
\[
\|f\|_{X^{\pm, b, r}_{k,p,q}} = \left( \sum_{j \succeq -k} \| Q^\pm_j f\|^r_{X^{\pm, b}_{k,p,q}} \right)^\frac1r,
\]
with the usual modification when $r=\infty$.  We remark that $\| \cdot \|_{X^{\pm, b, r}_{k,p,q}}$ is a semi-norm since functions at very low modulations have zero output when measured this way. 

We note that if $p=q=2$ and $r=\infty$, this resembles the standard structure $X^{\pm, b, \infty}$ introduced above for $b=\frac12$, but there are some important differences. There is the threshold $j \succeq -k$ and there is the cap structure when $j \leq k$, none of which appears in $X^{\pm, \frac12, \infty}$; the cap structure is less of an issue due to Plancherel, but this is not the case for other pairs $(p,q)$.

 In this paper we mostly work with $X^{\pm, b, r}_{k,p,q}$ with the choice $(p,q)=(\frac43,4)$ and its dual pair $(p,q)=(4,\frac43)$, and $r=1$ and $r=\infty$. 

Using these structures brings the following challenges: if $j \prec k$, one does not have uniform $L^1_{t,x}$ bounds on kernel of $Q_j^\pm$ and the support of the set $|\tau \pm \la \xi \ra| \approx 2^j$ becomes sensitive to the curvature properties of the characteristic surface $\tau = \mp \la \xi \ra$; 
this is not a problem in the context of $\dot X^{\pm,\frac12,\infty}$ where one can rely on Plancherel instead of kernel bounds. In Lemma \ref{kerb} below we investigate such bounds.

\begin{lem}\label{kerb} 
Assume that $j \leq k$, $2l=k-j$ and $\kappa \in \mathcal{K}_l$. Then the kernels  $K^{\pm}_{k,j,\kappa}$ and $\bar K^{\pm}_{k,j,\kappa}$ of the multipliers $P_\kappa Q_j^\pm P_k$, respectively  $\frac{ P_\kappa Q_j^\pm P_k}{-i \partial_t \pm \la D \ra}$, satisfy the following bounds
\begin{equation}
\| K^\pm_{k,j,\kappa} \|_{L^1(\R^3)} \ls 1, \quad \| \bar K^{\pm}_{k,j,\kappa} \|_{L^1(\R^3)} \ls 2^{-j}
\end{equation}
where the constant used in $\ls$ is independent of $k,j,\kappa$. 
\end{lem}

We postpone the proof of this Lemma to the end of the section. 

The definition of $X^{\pm, b, r}_{k,p,q}$ contains a certain rigidity that relates the scale of the caps $\kappa$ to that of the frequency and modulation through the relation $2l=k-j$. This rigidity stems from Lemma \ref{kerb}. In our analysis, we need to use estimates at the level of $X^{\pm, b, r}_{k,p,q}$ with a mismatched cap size. To formalize this, we fix $k$ and $j$ with $j \leq k$. If $l=\frac{k-j}2$, then the following holds true
\[
\left( \sum_{\kappa \in \mathcal{K}_{l}} \| P_\kappa Q^\pm_j f \|_{L^p_t L^q_x}^2 \right)^\frac12 \ls 2^{-jb} \| Q^\pm_j f\|_{X^{\pm, b}_{k,p,q}}. 
\]
This is easily seen to hold true if $l$ is replaced by $l'$ with $|l-l'| \ls 1$. Assume now that $l' \leq l$. In that case any cap $k' \in \mathcal{K}_{l'}$ contains $\approx 2^{l-l'}$ caps $k \in \mathcal{K}_{l}$, thus we obtain
\begin{equation} \label{Kll}
\| P_{\kappa'} Q^\pm_j f \|_{L^p_t L^q_x} \ls 2^{\frac{l-l'}2} \left( \sum_{\kappa \in \mathcal{K}_{l}: \kappa \cap \kappa' \ne \emptyset} \| P_\kappa Q^\pm_j f \|_{L^p_t L^q_x}^2 \right)^\frac12.
\end{equation}
Assume now that $l \leq l'$. In this case we use the trivial estimate 
$\| P_{\kappa'}  P_\kappa Q^\pm_j f \|_{L^p_t L^q_x} \ls \| P_\kappa Q^\pm_j f \|_{L^p_t L^q_x}$ (which follows from the uniform boundedness of the kernel of $P_{\kappa'}$ in $L^1_x$) and the fact that each cap $\kappa \in \mathcal{K}_{l}$
contains $\approx 2^{l'-l}$ caps $\kappa' \in \mathcal{K}_{l'}$ to obtain
\begin{equation} \label{Kll2}
\left( \sum_{\kappa' \in \mathcal{K}_{l'}: \kappa \cap \kappa' \ne \emptyset} \| P_{\kappa'} P_{\kappa} Q^\pm_j f \|_{L^p_t L^q_x}^2 \right)^\frac12 \ls 2^{\frac{l'-l}2}  \| P_\kappa Q^\pm_j f \|_{L^p_t L^q_x}.
\end{equation}
Thus we obtain the following estimate
\begin{equation} \label{Kllg}
\left( \sum_{\kappa' \in \mathcal{K}_{l'}} \| P_{\kappa'} Q^\pm_j f \|_{L^p_t L^q_x}^2 \right)^\frac12 \ls 2^{\frac{|l-l'|}2} \left( \sum_{\kappa \in \mathcal{K}_{l}} \| P_\kappa Q^\pm_j f \|_{L^p_t L^q_x}^2 \right)^\frac12 \ls 2^{\frac{|l-l'|}2} 2^{-jb} \| Q^\pm_j f\|_{X^{\pm, b}_{k,p,q}}, 
\end{equation}
which holds true regardless of the relation between $l$ and $l'$. We also record a particular case of this corresponding to $l'=0$:
\begin{equation} \label{esayX}
\| Q^\pm_j f \|_{L^p_t L^q_x} \ls 2^{\frac{k-j}4} 2^{-jb} \| Q^\pm_j f\|_{X^{\pm, b}_{k,p,q}}. 
\end{equation}

We now return to the construction of our resolution space, which is essentially based on information at the level of $S_k^\pm$ and $X^{\pm, \frac12, 1}_{k,\frac43,4}$.
There is one important observation: just as the structures that are found in $S_k$ come with various levels of losses in regularity, see the factors of $2^{-\frac{k}4}$ and $2^{-\frac{3k}4}$ that are used to compensate for that loss, a similar loss is encountered in the structure  $X^{\pm, \frac12, 1}_{k,\frac43,4}$. While the losses mentioned in the Strichartz estimates are clearly seen at the level of linear estimates, see Proposition \ref{lem:full-str}, the loss recorded in the structure $X^{\pm, \frac12, 1}_{k,\frac43,4}$ is a byproduct of the nonlinear analysis below; for now we simply denote the loss by $r_0$, and find its value once the nonlinear analysis is completed. Thus we define
\begin{equation} \label{Zdef}
\| f\|_{Z_k^\pm} =  \| f \|_{S_k^\pm} + 2^{-\rl k} \|f\|_{X^{\pm, \frac12, 1}_{k,\frac43,4}}. 
\end{equation}
Next, we consider boundedness properties of certain multipliers in the context of our spaces.

\begin{lem}
  \label{lem:stable} {\it i)} Let $s_1,s_2\in \{+,-\}$.  For any
  $k_1,k_2\in \N_0$, $1\leq l\leq \min(k_1,k_2)+10$,
  $\kappa_1,\kappa_2\in \mathcal{K}_l$ with
  $\dist(s_1\kappa_1,s_2\kappa_2)\ls 2^{-l}$, $v_1,v_2\in \C^2$, we
  have
  \begin{equation}\label{eq:stable1}
    |\la \Pi_{s_1}(2^{k_1}\omega(\kappa_1))v_1,\beta\Pi_{s_2}(2^{k_2}\omega(\kappa_2))v_2\ra |\ls 2^{-l}|v_1||v_2|.
  \end{equation}

  {\it ii)} Fix $k \in \N_0$ and assume that $f$ is localized at frequency $2^k$, i.e. it satisfies $f =\tilde{P}_{k} f$. For any $1\leq l\leq k+10$, $\kappa \in
  \mathcal{K}_l$, $f \in S^{\pm}_{k}$, we have
  \begin{equation}\label{eq:stable2}
    \|[\Pi_{\pm}(D)-\Pi_{\pm}(2^{k}\omega(\kappa))]P_{\kappa} f \|_{S^{\pm}_{k}}\ls 2^{-l}\|P_{\kappa} f \|_{S^{\pm}_{k}},
  \end{equation}
and similarly in any $L^p_tL^q_x$-norms (not necessarily Schr\"odinger-admissible).

\end{lem}

This is essentially Lemma 3.3 from \cite{BH-DKG} adapted to our two dimensional setting and the argument is identical.

\subsection{The linear theory}

The next Lemma provides the standard estimates for the inhomogeneous linear equation, and it does so using a dual type formulation. 
\begin{lem}\label{lem:lin}
i) For any $k \in \N_0$, $u_0=\tilde{P}_ku_0\in L^2(\R^2;\C^d)$ and $f=\tilde{P}_kf\in L^1_t(\R,L^2(\R^2;\C^d))$, let
\[
u(t)=e^{\mp it \la D\ra}u_0 +i \int_0^t e^{\mp i(t-s) \la D\ra}f(s)ds.
\]
Then, $u=\tilde{P}_k u$ is the unique solution of
\[
-i\partial_t u \pm \la D \ra u=f,
\]
and $u\in C(\R,L^2(\R^2;\C^d))$, and
\begin{equation}\label{eq:lin}
\|u\|_{S^\pm_k}\ls{} \|u_0\|_{L^2(\R^2)}+\sup_{g \in G}\Big|\int_{\R^{1+2}}\la f, g\ra_{\C^d} dxdt \Big|
\end{equation}
provided that the right hand side of \eqref{eq:lin} is finite, where $G$ is defined as
the set of all $g=\tilde{P}_k g\in L^\infty_t (\R;L^2(\R^2;\C^d))$ such that $\|g\|_{S^{\pm}_k}=1$.

ii) With the above setup we also have
\begin{equation}\label{eq:lin2}
\|u\|_{X^{\pm, \frac12, 1}_{k,\frac43,4}}\ls{} \sup_{j \succeq -k} \sup_{g \in G_j}\Big|\int_{\R^{1+2}}\la f, g\ra_{\C^d} dxdt \Big|
\end{equation}
provided that the right hand side of \eqref{eq:lin2} is finite, where $G_j$ is defined as
the set of all $g=\tilde{P}_k \tilde Q_j^\pm g\in L^\infty_t (\R;L^2(\R^2;\C^d))$ such that $\|g\|_{X^{\pm, \frac12, \infty}_{k,4,\frac43}}=1$.

\end{lem}

Part i) of the Lemma is essentially Lemma 3.4 from \cite{BH-DKG} adapted to our two dimensional setting and the argument is almost identical; the only add-on feature that we use here is the stability of the estimates under modulation cut-offs, but this again follows from the stability of the $V^2_{\pm \la D \ra}$ structure under modulation cut-offs. For convenience, we provide a sketch of the argument below. Part ii) is new and it is tailored around the use of the new structure  $X^{\pm, \frac12, 1}_{k,\frac43,4}$ introduced in this paper.

\begin{proof}
i) Following the argument in Lemma 3.4 from \cite{BH-DKG}, we argue using the theory of $U^2$ and $V^2$ spaces, see e.g.\ \cite{KT,HHK,KVT} for details.
We recall that for $1<p<\infty$ the atomic space $U^p_{\pm\la D\ra}$ is defined via its atoms
\[
a(t)=\sum_{k=1}^K \mathds{1}_{[t_{k-1},t_{k})}(t)e^{\mp i t \la D\ra} \phi_k , \quad \sum_{k=1}^K\|\phi_k\|_{L^2}^p=1,
\]
where $\{t_k\}$ is a partition, $t_K=+\infty$.

As a companion space we use the space $V^p_{\pm\la D\ra}$ of right-continuous functions $v$ such that $t\mapsto e^{\pm i t \la D\ra}v(t)$ is of bounded $p-$variation. We have $ U^2_{\pm\la D\ra}\hookrightarrow V^2_{\pm\la D\ra}\hookrightarrow U^p_{\pm\la D\ra}$ for $p>2$.

For $0\leq l \leq k$ we define
\begin{equation}\label{eq:loc-u2}
\|u\|_{U^\pm_{k;l}}:=\Big(\sum_{\kappa \in \mathcal{K}_l} \| P_{\kappa} u \|^2_{U^2_{\pm\la D\ra}}\Big)^{\frac12}.
\end{equation}
Then, we have
\begin{equation}\label{eq:loc-v2}
\|u\|_{V^\pm_{k;l}}:=\Big(\sum_{\kappa \in \mathcal{K}_l} \| P_{\kappa} u \|^2_{V^2_{\pm\la D\ra}}\Big)^{\frac12}\ls \|u\|_{U^\pm_{k;l}}.
\end{equation}
It is easy to show that the $U^\pm_{k;l}$-norms are decreasing if we localize to smaller scales, i.e.\
\[
\|u\|_{U^\pm_{k;l}}\ls \|u\|_{U^\pm_{k;\tilde{l}}} \text{ if } \tilde{l}\leq l,
\]
and the $V^\pm_{k;l}$-norms are increasing if we localize to smaller scales, i.e.\
\[
\|u\|_{V^\pm_{k;l}}\ls \|u\|_{V^\pm_{k;\tilde{l}}} \text{ if } \tilde{l}\geq l.
\]
Set $U^\pm_{k}=U^\pm_{k;k}$ and $V^\pm_{k}=V^\pm_{k;k}$.

Strichartz estimates for admissible pairs $(p,q)$ hold for $U^p_{\pm\la D\ra}$-functions (which is easily verified for atoms), hence all for $V^2_{\pm\la D\ra}$-functions. For any $0\leq l\leq k$ we have
\[
2^{-\frac{2k}{p}} \Big(\sum_{\kappa \in \mathcal{K}_l} \| P_{\kappa}  u \|^2_{L^p_tL^q_x}\Big)^{\frac12}
\ls{} \|u\|_{V^\pm_{k;l}}\ls \|u\|_{V^\pm_{k}}.
\]
Since $\|Q_{\leq j} u \|_{V_k^{\pm}} \ls \|u \|_{V_k^{\pm}}$ for all $j \in \Z$ (see e.g. \cite[Corollary ~2.18]{HHK}) the above inequality holds true if $u$ is replaced by $Q_{\leq j} u$.  We also have $V^\pm_{k}\hookrightarrow V^2_{\pm\la D\ra}$ and $V^2_{\pm\la D\ra}$-norm dominates both the $L^\infty_t L^2_x$-norm and the $\dot{X}^{\pm,\frac12,\infty}$-seminorm. Hence,
\[\|u\|_{S^\pm_k}\ls \|u\|_{V^\pm_{k}}\ls \|u\|_{U^\pm_{k}}.\]
Now, we can use the $U^2$ duality theory (see e.g. \cite[Prop.~2.10]{HHK}, and \cite[Prop.~2.11]{HTT} for a frequency-localized version), to conclude that
\begin{equation} \label{Ukest}
    \|u\|_{U^\pm_k}\ls{}\|u_0\|_{L^2(\R^2)}+\sup_{h \in H}\Big|\int_{\R^{1+2}}\la f, h\ra_{\C^d} dxdt \Big|,
\end{equation}
where $H$ is defined as the set of all $h=\tilde{P}_k h$ such that $\|h\|_{V^\pm_k}=1$.
The claim now follows by using again $\|g\|_{S^{\pm}_k}\ls \|g\|_{V^\pm_k}$.

\begin{rmk}\label{rmk:u2} From \eqref{Ukest}, it follows that  we can upgrade \eqref{eq:lin} to estimating  $\|u\|_{U^{\pm}_{k}}$ by the same quantity (instead of the current $\|u\|_{S^\pm_{k}}$).
\end{rmk}

ii) Fix $j \succeq -k$. A simple duality argument shows that 
\[
\| Q_j f\|_{X^{\pm, - \frac12}_{k,\frac43,4}} \ls \sup_{g \in G_j}\Big|\int_{\R^{1+2}}\la f, g\ra_{\C^d} dxdt \Big|.
\]
Next let $2l=k-j$ and $\kappa \in \mathcal{K}_l$; we apply the modulation and angular projectors to the equation to obtain
 \[
  (-i\partial_t \pm \langle D \rangle) Q_j P_\kappa u = Q_j P_\kappa f,
 \]
from which 
 \[
 Q_j P_\kappa u = \frac{Q_j P_\kappa f}{ -i\partial_t \pm \langle D \rangle}. 
 \]
From Lemma \ref{kerb} we obtain the desired bound for fixed $\kappa$, that is
\[
\| Q_j P_\kappa u \|_{L^\frac43_t L^4_x} \ls 2^{-j} \|Q_j P_\kappa f\|_{L^\frac43_t L^4_x}. 
\]
Then the $\ell^2$ structure with respect to caps is trivially inherited from that of $f$, thus leading to
\[
\| Q_j u\|_{X^{\pm, \frac12}_{k,\frac43,4}} \ls \| Q_j f\|_{X^{\pm, - \frac12}_{k,\frac43,4}} \ls \sup_{g \in G_j}\Big|\int_{\R^{1+3}}\la f, g\ra_{\C^d} dxdt \Big|.
\]
Finally, the estimate in $X^{\pm, \frac12,1}_{k,\frac43,4}$, which performs the summation with respect to $j$, is recovered from the duality between $l^1_j$ and $l^\infty_j$. This finishes the argument for \eqref{eq:lin2}.

\end{proof}

Our resolution space $S^{\pm,\sigma}$ corresponding the Sobolev regularity $\sigma$ will be the space of functions in $C(\R,H^\sigma(\R^2;\C^d))$ such that
\[
\| f \|_{S^{\pm,\sigma}} = \| P_{\leq 0} f \|_{S_{\leq 0}^\pm} +
\left( \sum_{k \geq 1} 2^{2\sigma k} \| P_k f \|_{S_k^\pm}^2
\right)^\frac12<+\infty,
\]
which is obviously a Banach space. Similarly, we define $Z^{\pm,\sigma}$ such that
\[
\| f \|_{Z^{\pm,\sigma}} = \| P_{\leq 0} f \|_{Z_{\leq 0}^\pm} +
\left( \sum_{k \geq 1} 2^{2\sigma k} \| P_k f \|_{Z_k^\pm}^2
\right)^\frac12<+\infty.
\]

\begin{proof}[Proof of Lemma \ref{kerb}] We establish the estimate for $\bar K^{\pm}_{k,j,\kappa}$ the kernel of the multiplier $\frac{P_\kappa Q_j^\pm P_k}{-i \partial_t \pm \la D \ra}$; the argument for $K^\pm_{k,j,\kappa}$ follows in a similar, but simpler, way. Also, without restricting the generality of the argument, it suffices to provide the argument for the $-$ choice in the multiplier; the argument in the $+$ case is identical.  The multiplier  $\frac{P_\kappa Q_j^- P_k}{-i \partial_t - \la D \ra}$ has symbol $\frac{\rho_j(\tau - \la \xi \ra) \eta_\kappa(\xi) \rho_k(\xi)}{\tau - \la \xi \ra} $. 
The relation between the scales of this localization, that is $2l=k-j$, has been chosen such that the support of $\rho_j(\tau - \la \xi \ra) \eta_\kappa(\xi) \rho_k(\xi)$ is well approximated by a rectangular parallelepiped and at the same time we have, the obvious, approximate constancy of the denominator, that is $\tau - \la \xi \ra \approx 2^{j}$.  This motivates changing the coordinates from $(\tau,\xi_1,\xi_2)$ to coordinates which are better aligned with the directions principal directions of the 
 above mentioned parallelepiped. This change is not necessary at high modulations that is when $j \geq k-10$; in this case a similar but simpler computation in the original coordinates works just fine. Thus bellow we assume $j \leq k-10$. 
 
 To keep notation simple, we can assume that $\omega(\kappa)=(1,0)$.  We define the new coordinates by $\tilde \tau=\frac1{\sqrt{2}} (\tau-\xi_1), \tilde \xi_1=\frac1{\sqrt{2}} (\tau+\xi_1), \tilde \xi_2=\xi_2$.  We know that $a$ is supported in a rectangular parallelepiped of size $C 2^{-j} \times C 2^k \times C 2^{k-l}$, where the sizes are in the directions $\tilde \tau, \tilde \xi_1, \xi_2$ respectively; for $z \in \{\tilde \tau, \tilde \xi_1, \xi_2\}$ we let $s(z)$ be the size of the parallelepiped in the $z$ direction.
Our main claim here is that $a$ satisfies the bounds
 \begin{equation} \label{abound}
 |\partial_z^\alpha a| \ls_\alpha 2^{-j} (s(z))^{-\alpha},
 \end{equation}
 for any $z \in \{\tilde \tau, \tilde \xi_1, \xi_2\}$ and $0 \leq \alpha \leq 4$. 
 This suffices to establish that $K=\check a \in L^1(\R^3)$ with uniform bounds,
thus proving the Lemma. 
 
In the remaining of this proof we provide a full argument for \eqref{abound}.

 Let $z=\tilde \tau$.  The strategy is simple. Since  $a= \frac{\rho_j(\tau - \la \xi \ra) \eta_\kappa(\xi) \rho_k(\xi)}{\tau - \la \xi \ra} $ is a product of symbols, it suffices to estimate each term. All the estimates below are made under the assumption that $(\tau,\xi_1,\xi_2)$ belong to the support of $a$. 

Using the formula  $ (\xi_1,\xi_2)= (\frac{1}{\sqrt{2}}(\tilde \xi_1 - \tilde \tau), \xi_2)$, a direct computation gives $| \partial^\alpha_{\tilde \tau}  |\xi| |\ls_\alpha 2^{k(1-\alpha)}$, from which it follows that
\begin{equation} \label{ab1}
| \partial^\alpha_{\tilde \tau}  \rho_k(|\xi|) |\ls_\alpha 2^{-k \alpha} \ls 2^{-j \alpha}, 
\end{equation}
where we have used the fact that $ j \leq k-10$. 
 
 Given our choice $\omega=(1,0)$, it follows that $\angle (\xi)$, the angle that $\xi$ makes with $\omega(\kappa)$ (the center of the cap), is given by
 \[
 \angle (\xi)= \arcsin \left(\frac{\xi_2}{\xi_1}\right)= \arcsin \left(\frac{\xi_2}{\frac{1}{\sqrt{2}}(\tilde \xi_1 - \tilde \tau),}\right);
 \]
as a consequence $\eta_\kappa(\xi)= \rho_{-l} (\angle (\xi))$. Next we estimate
$ |  \partial^\alpha_{\tilde \tau}  \arcsin(\frac{\xi_2}{\xi_1})  |\ls_\alpha |\frac{\xi_2}{\xi_1^{1+\alpha}}| \ls 2^{-l} 2^{-\alpha k} $;
 in the later one we use the fact that $|\frac{\xi_2}{\xi_1}| \leq \frac14$. From this it follows that
 \begin{equation} \label{ab2}
 |  \partial^\alpha_{\tilde \tau}  \eta_\kappa(\xi)| \ls 2^{-\alpha k} \ls 2^{-j \alpha}. 
 \end{equation}
 
 Just as in the proof of \eqref{ab1}, we obtain $| \partial^\alpha_{\tilde \tau}  \la \xi \ra |\ls_\alpha 2^{k(1-\alpha)}$. Since $\partial_{\tilde \tau} \tau=\frac1{\sqrt{2}}$
 and $\partial^\alpha_{\tilde \tau} \tau = 0, \forall \alpha \geq 2$, it follows that $ | \partial^\alpha_{\tilde \tau}  (\tau-\la \xi \ra) | \ls_\alpha 2^{k(1-\alpha)} $. From this we obtain two estimates
 \begin{equation} \label{ab3}
  | \partial^\alpha_{\tilde \tau}  \rho_j (\tau-\la \xi \ra) | \ls_\alpha 2^{-j \alpha},
 \end{equation}
 and 
 \begin{equation} \label{ab4}
  | \partial^\alpha_{\tilde \tau}  (\tau-\la \xi \ra)^{-1} | \ls_\alpha 2^{-j (\alpha+1)};
 \end{equation}
 in the above we use the chain rule an the fact that $2^{-k} \ls 2^{-j}$. Using \eqref{ab1}, \eqref{ab2}, \eqref{ab3} and \eqref{ab4}, we obtain \eqref{abound} for 
 $z=\tilde \tau$. 
 
 The proof for $z=\tilde \xi_1$ is similar. Indeed, a quick inspection of the arguments above, establishes that
 \begin{equation} \label{ab12}
 | \partial^\alpha_{\tilde \xi_1}  \rho_k(|\xi|) |\ls_\alpha 2^{-k \alpha}, \quad |  \partial^\alpha_{\tilde \xi_1}  \eta_\kappa(\xi)| \ls 2^{-\alpha k}.
 \end{equation}
 We then write
 \[
 \begin{split}
 \tau - \sqrt{\xi_1^2+\xi^2_2+1}  &= \frac1{\sqrt{2}} (\tilde \tau + \tilde \xi_1) - \sqrt{\frac12(\tilde \xi_1 - \tilde \tau)^2+\xi^2_2+1} \\
  &=  \frac{\frac12 (\tilde \tau + \tilde \xi_1)^2 -  (\frac12(\tilde \xi_1 - \tilde \tau)^2+\xi^2_2+1)}{\frac1{\sqrt{2}} (\tilde \tau + \tilde \xi_1) + \sqrt{\frac12(\tilde \xi_1 - \tilde \tau)^2+\xi^2_2+1}}\\  
 &= \frac{2 \tilde \tau \cdot \tilde \xi_1 -  (\xi^2_2+1)}{\frac1{\sqrt{2}} (\tilde \tau + \tilde \xi_1) + \sqrt{\frac12(\tilde \xi_1 - \tilde \tau)^2+\xi^2_2+1}} \\
 &= \frac{2 \tilde \tau \cdot \tilde \xi_1 -  (\xi^2_2+1)}{\frac1{\sqrt{2}} (\tau + \la \xi \ra) }. 
 \end{split}
 \]
 We note that the denominator is in fact (in the original coordinates) $ \tau + \sqrt{\xi_1^2+\xi^2_2+1} =  \tau - \sqrt{\xi_1^2+\xi^2_2+1}  + 2 \sqrt{\xi_1^2+\xi^2_2+1}
 \approx 2^k$ given that $ |\tau - \sqrt{\xi_1^2+\xi^2_2+1}| \leq 2^{j+1} \leq 2^{k-9} < 2^{-5} \sqrt{\xi_1^2+\xi^2_2+1}$ within the support of $\rho_j(\tau - \la \xi \ra )\rho_k(\xi)$. The above computation provides  the formula for $\tau - \la \xi \ra$ in the new coordinates.

 Using the formula $\tau+\la \xi \ra=\frac1{\sqrt{2}}(\tilde \tau + \tilde \xi_1) + \sqrt{\frac12(\tilde \xi_1 - \tilde \tau)^2+\xi^2_2+1}$ and the observation that $\la \xi \ra \approx 2^k, |\tilde \tau| \ls 2^k, |\tilde \xi_1| \ls 2^k$ in the region where we estimate it is straightforward to establish that
 $| \partial_{\tilde \xi_1}^\alpha (\tau + \la \xi \ra) | \ls_\alpha  2^{k(1-\alpha)}$. Based on this, the above formula, and the fact that $j \leq k-10$ we obtain
 \[
 | \partial^\alpha_{\tilde \xi_1} (\tau - \la \xi \ra) | \ls 2^{j-\alpha k}.
 \]
 From this, using the chain and product rule, we obtain
 \begin{equation} \label{ab13}
 | \partial^\alpha_{\tilde \xi_1} \rho_j (\tau - \la \xi \ra) | \ls 2^{-\alpha k}, \quad  | \partial^\alpha_{\tilde \xi_1}  (\tau - \la \xi \ra)^{-1} | \ls 2^{-j-\alpha k}.
 \end{equation}
 From \eqref{ab12}, \eqref{ab13} we obtain \eqref{abound} for $z=\tilde \xi_1$. 
  
 Finally we prove \eqref{abound} for $z=\xi_2$.  Just as before $\partial^\alpha_{\xi_2} |\xi | \ls  2^{-k(1-\alpha)}$;
 from this we obtain
 \begin{equation} \label{ab1xi2}
  | \partial^\alpha_{\tilde \xi_2}  \rho_k(|\xi|) |\ls_\alpha 2^{-k\alpha} \ls 2^{-(k-l) \alpha},
 \end{equation}
 Next we estimate $ |  \partial^\alpha_{\xi_2}  \arcsin(\frac{\xi_2}{\xi_1})  |\ls_\alpha  2^{-\alpha k}, \forall \alpha \geq 1$. From this it follows that
   \begin{equation} \label{ab2xi2}
|  \partial^\alpha_{ \xi_2}  \eta_\kappa(\xi)| \ls 2^{-\alpha (k-l)}.
 \end{equation}
 We notice that $\partial^\alpha_{\xi_2} \la \xi \ra \ls  2^{-k(1-\alpha)}, \alpha \geq 0$, but this is not be enough to derive the estimates involving $\rho_j(\tau - \la \xi \ra)$ and $ (\tau - \la \xi \ra)^{-1}$. Instead we notice that $\partial_{\xi_2} \la \xi \ra \ls  2^{-l}$ and $\partial_{\xi_2}^3 \la \xi \ra \ls  2^{-l-2k}$. From these estimates and the relation $2l=k-j$ we obtain
 \begin{equation} \label{ab7xi2}
  | \partial^\alpha_{\tilde \xi_2}  (\tau - \la \xi \ra)^{-1} |\ls_\alpha 2^{-j} 2^{-(k-l)\alpha}, \quad 
  | \partial^\alpha_{\tilde \xi_2}  \rho_j(\tau - \la \xi \ra) |\ls_\alpha 2^{-(k-l)\alpha},
 \end{equation}
for $\alpha \leq 4$. This finishes the argument for \eqref{abound}. We note that the most efficient way to carry the computation above is using the Fa\`{a} di Bruno's formula: 
\[
\partial_x^n f(g(x)) =  \sum \frac{n ! }{m_1 ! (1 !)^{m_1} m_2 ! (2 !)^{m_2} ... m_n ! (n !)^{m_n} } f^{(m_1+...+m_n)} \Pi_{j=1}^n (g^{(j)})^{m_j},
\]
 where the sum is performed over the $n$-tuples of non-negative integers $(m_1,..,m_n)$ subject to $1 \cdot m_1 + 2 \cdot m_2 + ... + n \cdot m_n=n$.

\end{proof}

\section{Nonlinear estimates and the proof of the main result}\label{sect:nl}

In this section we provide the estimates for the nonlinear terms in \eqref{DKGf}, which is our reformulation of the DKG system. These estimates are guided by Lemma \ref{lem:lin}, where $f$ should be thought as the nonlinear term which we want to estimate. There will be two sets of nonlinear estimates: one that recovers the structure $S_k^\pm$ (for the solutions of \eqref{DKGf}) and is based on \eqref{eq:lin}, and one that recovers the structure $X_{k,\frac43,4}^{\pm,\frac12,1}$ and is based on \eqref{eq:lin2}.

Recall \eqref{DKGf} with the convention $M=m=1$ and use the decomposition $\psi=\Pi_+(D)\psi+\Pi_-(D)\psi$ in the nonlinearities that appear in \eqref{DKGf} (for all three terms). 

The first set of estimates is the following
\begin{align*}
  & \Big|\int \la \Pi_{s_2}(D)[\Re\phi \, \beta \Pi_{s_1}(D)\psi_1],\psi_2\ra dxdt\Big| 
  \ls  \|\phi\|_{Z^{+,r+\frac12}}\|\psi_1\|_{Z^{s_1,r}}\|\psi_2\|_{S^{s_2,-r} }\\
 & \Big|\int \la D \ra^{-1}\la \Pi_{s_1}(D)\psi_1, \beta \Pi_{s_2}(D)\psi_2\ra \, \overline{\phi} dxdt\Big| 
 \ls   \|\phi\|_{S^{+,-r-\frac12} }\|\psi_1\|_{Z^{s_1,r}}\|\psi_2\|_{Z^{s_2, r}},
\end{align*}
for any choice of signs $s_1,s_2\in \{+,-\}$. The first inequality estimates the nonlinearity that appears in the equations of $\psi_\pm$, while the second estimates the nonlinearity that appears in the equation of $\phi$, see \eqref{DKGf}. These estimates follow from the more compact estimate
\begin{equation}\label{eq:red-est}
  \begin{split}
    &\Big|\int \phi \, \la \Pi_{s_1}(D)\psi_1,\beta\Pi_{s_2}(D)\psi_2\ra dxdt\Big|\\
    \ls{} &
\min \{   \|\phi\|_{Z^{+,r+ \frac12}} \|\psi_1\|_{Z^{s_1,r}} \|\psi_2\|_{S^{s_2,-r}} ,   \|\phi\|_{S^{+,-r+\frac12} }\|\psi_1\|_{Z^{s_1,r}}\|\psi_2\|_{Z^{s_2, r}} \}
  \end{split}
\end{equation}
 More precisely, we will prove this  first on the dyadic level, where all integrals are clearly finite, cp.\ Lemma \ref{lem:lin}.

The second set of estimates is the following
\begin{align*}
  & \Big|\int \la \Pi_{s_2}(D)[\Re\phi \, \beta \Pi_{s_1}(D)\psi_1],\psi_2\ra dxdt\Big| 
  \ls  \|\phi\|_{S^{+,r+\frac12}}\|\psi_1\|_{S^{s_1,r}}\|\psi_2\|_{ X^{s_2, -r +\rl ,\frac12, \infty}_{4,\frac43}}\\
 & \Big|\int \la D \ra^{-1}\la \Pi_{s_1}(D)\psi_1, \beta \Pi_{s_2}(D)\psi_2\ra \, \overline{\phi} dxdt\Big| 
 \ls   \|\phi\|_{ X^{+, -r -\frac12 +\rl ,\frac12, \infty}_{4,\frac43}}\|\psi_1\|_{S^{s_1,r}}\|\psi_2\|_{S^{s_2, r}},
\end{align*}
for any choice of signs $s_1,s_2\in \{+,-\}$. These estimates follow from
 \begin{equation}\label{eq:red-est-2}
  \begin{split}
    &\Big|\int \phi \, \la \Pi_{s_1}(D)\psi_1,\beta\Pi_{s_2}(D)\psi_2\ra dxdt\Big|\\
    \ls{} &
\min \{   \|\phi\|_{S^{+,r+ \frac12}} \|\psi_1\|_{S^{s_1,r}}   \|\psi_2\|_{ X^{s_2, -r +\rl,\frac12, \infty}_{4,\frac43}}    ,    \|\phi\|_{ X^{+, -r +\frac12 +\rl ,\frac12, \infty}_{4,\frac43}}\|\psi_1\|_{S^{s_1,r}}\|\psi_2\|_{S^{s_2, r}} \},
  \end{split}
\end{equation}
for any choice of signs $s_1,s_2\in \{+,-\}$. Our main result in this section is the following:
\begin{pro} \label{mainpro}
If $r > \frac12$ and $r_0=1$, then \eqref{eq:red-est} and \eqref{eq:red-est-2} hold true.
\end{pro}

Below we proceed with the proof of this Proposition which is essentially the analysis of the two estimates above \eqref{eq:red-est} and \eqref{eq:red-est-2}. This is split in two subsections, one for each estimate, due to the length of the arguments.

In both subsections our aim will be to rewrite \eqref{eq:red-est} and \eqref{eq:red-est-2} for functions localized in frequency. The following class of functions will play an important role in our analysis: $G: \N^3_0 \rightarrow (0,\infty)$ such that
\begin{equation} \label{G} \sum_{k,k_1,k_2 \in \N_0 \atop \max(k,k_1,k_2)\sim \med(k,k_1,k_2)}
  \frac{G(k,k_1,k_2)a_{k}
  b_{k_1} c_{k_2} }{(\min(k,k_1,k_2)+1)^{10}} \ls \| a \|_{l^2} \| b \|_{l^2} \| c \|_{l^2} 
\end{equation}
for all sequences $a=(a_j)_{j \in \N_0}$ etc.\ in $l^2(\N_0)$. We write $\mathbf{k}=(k,k_1,k_2)$.

\subsection{Estimates for dyadic pieces - part 1} 
Our main result in this section is the following.

\begin{pro} \label{PTR} Let $s_1,s_2\in\{+,-\}$. There exists a
  function $G: \N^3_0 \rightarrow (0,\infty)$ such that for all
  $\phi=P_{k}\phi,\psi_i=P_{k_i}\Pi_{s_i}(D)\psi_i$, $i=1,2$, the
  following estimate holds true:
  \begin{equation} \label{cunn2} 
  \begin{split}
  & \Big| \int \phi \la \psi_1, \beta
    \psi_2 \ra dx dt\Big| \\
    \ls{} & G(\mathbf{k}) 2^{\frac{k}2}
 \min ( 2^{ r (-k+k_1+k_2)} \| \phi \|_{S^{+}_{k}}  \| \psi_1 \|_{Z^{s_1}_{k_1}} \| \psi_2 \|_{Z^{s_2}_{k_2}},  
 2^{  r (k-k_1+k_2)}  \| \phi \|_{Z^{+}_{k}}  \| \psi_1 \|_{S^{s_1}_{k_1}} \| \psi_2 \|_{Z^{s_2}_{k_2}} ).
 \end{split}
  \end{equation}
Assuming that $r \geq \frac12$ and $r+\frac12 \geq \rl $, the function $G$ satisfies \eqref{G}. 
  
\end{pro}

It is clear that \eqref{eq:red-est}, with the choice $r_0=1$ and $r > \frac12$, is implied by the above result. 

\begin{proof} We denote the integral on the left hand side of
  \eqref{cunn2} by $I(\mathbf{k})$.
 We note that in most of the cases (appearing throughout the argument) we are able estimate $I(\mathbf{k})$ by using only the $S^{\pm}$ type norms, and this simplifies 
  the accounting in \eqref{cunn2}. 
    
  We decompose $ I(\mathbf{k})=I_0(\mathbf{k})+I_1(\mathbf{k})+I_2(\mathbf{k})$,
  where
  \begin{align*}
    I_0(\mathbf{k}):=&\sum_{j\in \Z} \int Q_{j}^+ \phi \, \la Q^{s_1}_{\leq j}\psi_1, \beta Q^{s_2}_{\leq j}\psi_2 \ra dx dt, \ \ 
    I_1(\mathbf{k}):=\sum_{j_1\in \Z}\int Q_{<j_1}^+ \phi  \,\la Q^{s_1}_{j_1}\psi_1, \beta Q^{s_2}_{\leq j_1}\psi_2 \ra dxdt, \\
    I_2(\mathbf{k}):=&\sum_{j_2\in \Z}\int Q_{<j_2}^+ \phi  \,\la
    Q^{s_1}_{<j_2}\psi_1, \beta Q^{s_2}_{j_2}\psi_2 \ra dx dt.
  \end{align*}
  Each of these terms will be bounded individually $ | I_i(\mathbf{k})| \ls G_i(\mathbf{k}) \min (...)$,  where the $\min(...)$ is the same expression on the right-hand side of \eqref{cunn2} and we recover the bound on $G$ from
  $G(\mathbf{k}) \leq \sum_{i=1}^3 G_i(\mathbf{k})$.
   We split the argument into three cases.

\medskip

  {\bf Case 1:} $|k-k_2 | \leq 10$.

  {\it Contribution of $I_0(\mathbf{k})$:} We split
  $I_0(\mathbf{k})=I_{01}(\mathbf{k})+I_{02}(\mathbf{k})$ according to
  $j \prec k_1$ and $j\succeq k_1$. Then, due to Lemma \ref{lem:mod} there is
  no contribution if $j \prec k_1$ in the case $s_1=+,s_2=-$.  With all other choices of signs, we can restrict the sum in $I_{01}$ to $j \succeq -k_1$, so that by Lemma \ref{lem:mod} with $2l=k_1+k_2-k-j\sim k_1-j$ we have
\begin{align*}
I_{01}(\mathbf{k})= \sum_{-k_1 \preceq j \prec k_1} \sum_{\kappa_1,\kappa_2 \in \mathcal{K}_l \atop
      \dist(s_1 \kappa_1,s_2\kappa_2)\ls 2^{-l}} 
      \Big\{\int  Q^+_{ j } \phi \cdot \la P_{\kappa_1}Q_{\leq j}^{s_1} \psi_1, \beta  P_{\kappa_2} Q_{\leq j}^{s_2}  \psi_2 \ra dx dt\Big\}.
  \end{align*}
In earnest one should use the notation $l(j)$ instead of just $l$ above, to indicate that $l$ depends on $j$, which is clear from the formula $2l=k_1+k_2-k-j\sim k_1-j$; to keep notation compact we keep this dependence implicit.  In view of Lemma \ref{lem:stable}, we decompose
\[
\Pi_{s_i}(D)P_{\kappa_i}=[\Pi_{s_i}(D)-\Pi_{s_i}(2^{k_i}\omega(\kappa_i ))]P_{\kappa_i}+\Pi_{s_i}(2^{k_i}\omega(\kappa_i ))P_{\kappa_i},
\]
and obtain
  \begin{align*}
     \|\la P_{\kappa_1}Q_{\leq j}^{s_1} \psi_1, \beta  P_{\kappa_2}  Q_{\leq j}^{s_2}  \psi_2 \ra\|_{L^{2}_{t,x}}
\ls{}  2^{-l} \|P_{\kappa_1}Q_{ \leq j}^{s_1} \psi_1\|_{L^4_{t,x}} \| P_{\kappa_2}  Q_{\leq j}^{s_2}  \psi_2\|_{L^4_{t,x}},
  \end{align*}
  for any $\kappa_1,\kappa_2 \in \mathcal{K}_l$ with $ \dist(s_1 \kappa_1,s_2\kappa_2)\ls 2^{-l}$.  This type of argument, providing the gain from the null structure, will be repeated several times throughout the paper, and we will shortcut the formalization by simply stating only the last inequality above.

  Then, we apply Cauchy-Schwarz and perform the cap summation and obtain
  \begin{align*}
   | I_{01}(\mathbf{k})| & \ls \sum_{-k_1\preceq j \prec k_1} \|  Q^+_{ j } \phi \|_{L^2}  \sum_{\kappa_1,\kappa_2 \in \mathcal{K}_l \atop
      \dist(s_1 \kappa_1,s_2\kappa_2)\ls 2^{-l}} 2^{-l} \|P_{\kappa_1}Q_{ \leq j}^{s_1} \psi_1\|_{L^4_{t,x}} \| P_{\kappa_2}  Q_{\leq j}^{s_2}  \psi_2\|_{L^4_{t,x}} \\
  & \ls \sum_{-k_1\preceq j \prec k_1} 2^{-\frac{j}2} \|  Q^+_{ j } \phi \|_{X^{+, \frac12}}  2^{-l} \left( \sum_{\kappa_1 \in \mathcal{K}_l} \|P_{\kappa_1}Q_{ \leq j}^{s_1} \psi_1\|_{L^4_{t,x}}^2 \right)^\frac12 \left( \sum_{\kappa_2 \in \mathcal{K}_l} \| P_{\kappa_2}  Q_{\leq j}^{s_2}  \psi_2\|^2_{L^4_{t,x}} \right)^\frac12 \\
      &  \ls\sum_{-k_1\preceq j \prec k_1} 2^{-\frac{j}2}
    \| \phi \|_{S^+_{k}} 2^{-\frac{k_1-j}2} 2^{\frac{k_1}2} 
     \| \psi_1 \|_{S_{k_1}^{s_1}}
    2^{\frac{k_2}2}  \| \psi_2 \|_{S_{k_2}^{s_2}} \\
    & \ls \la k_1\ra 2^{\frac{k}2} \| \phi \|_{S^+_{k}}
    \| \psi_1 \|_{S^{s_1}_{k_1}} \| \psi_2 \|_{S^{s_2}_{k_2}}.
  \end{align*}
  
    In the range $j \succeq k_1$, a similar argument above with $l=0$,
  i.e.\ no cap decomposition and no gain from the null-structure,
  gives the bound $ | I_{02}(\mathbf{k})| \ls  
  2^{\frac{k}2} \| \phi \|_{S^+_{k}} \| \psi_1 \|_{S^{s_1}_{k_1}} \|  \psi_2 \|_{S^{s_2}_{k_2}}$.
  
Thus we obtain 
\begin{equation} \label{C1I0}
|G_0(\mathbf{k})| \ls \la k_1 \ra 2^{-rk_1}. 
\end{equation}

{\it Contribution of $I_1(\mathbf{k})$:}  We split  $I_1(\mathbf{k})=I_{11}(\mathbf{k})+I_{12}(\mathbf{k})$ according to
 $j_1 \prec k_1$ and $j_1\succeq k_1$. Again, by Lemma \ref{lem:mod} there is
  no contribution if $j_1 \prec k_1$ in the case $s_1=+,s_2=-$.  With all other choices of signs, we can restrict the sum in $I_{11}$ to $j_1\succeq -k_1$, so that by Lemma \ref{lem:mod} with $2l=k_1+k_2-k-j_1\sim k_1-j_1$ we have
\begin{align*}
I_{11}(\mathbf{k})= \sum_{-k_1\preceq j_1 \prec k_1} \sum_{\kappa_1,\kappa_2 \in \mathcal{K}_l \atop
      \dist(s_1 \kappa_1,s_2\kappa_2)\ls 2^{-l}} 
      \Big\{\int  Q^+_{\leq j_1 } \phi \cdot \la P_{\kappa_1}Q_{j_1}^{s_1} \psi_1, \beta  P_{\kappa_2} Q_{\leq j_1}^{s_2}  \psi_2 \ra dx dt\Big\}.
  \end{align*}
Just as above, we invoke Lemma \ref{lem:stable}, to obtain 
  \begin{align*}
     \|\la P_{\kappa_1}Q_{j_1}^{s_1} \psi_1, \beta  P_{\kappa_2}  Q_{\leq j_1}^{s_2}  \psi_2 \ra\|_{L^{\frac43}_{t,x}}
\ls{}  2^{-l} \|P_{\kappa_1}Q_{j_1}^{s_1} \psi_1\|_{L^2} \| P_{\kappa_2}  Q_{\leq j_1}^{s_2}  \psi_2\|_{L^4_{t,x}},
  \end{align*}
for any $\kappa_1,\kappa_2 \in \mathcal{K}_l$ with $ \dist(s_1 \kappa_1,s_2\kappa_2)\ls 2^{-l}$. By H\"older's inequality and Cauchy-Schwarz we obtain
  \begin{align*}
 | I_{11}(\mathbf{k})|  & \ls{} \sum_{-k_1\preceq j_1 \prec k_1} \Big\{ 2^{-\frac{k_1-j_1}2} \| Q^{s_1}_{j_1} \psi_1\|_{L^2}  \|Q^+_{<j_1} \phi\|_{L^4_{t,x}}  
  \cdot\Big( \sum_{\kappa_2 \in \mathcal{K}_l}\| P_{\kappa_2}  Q_{\leq j_1}^{s_2}  \psi_2\|_{L^4_{t,x}}^2 \Big)^{\frac12}\Big\}\\
& \ls{} \sum_{-k_1\preceq j_1<k_1} 2^{-\frac{k_1-j_1}2} 2^{-\frac{j_1}{2}} \|\psi_1\|_{S^{s_1}_{k_1}}2^{\frac{k}{2}} \|\phi\|_{S^+_{k}}2^{\frac{k_2}{2}} \|\psi_2\|_{S^{s_2}_{k_2}}\\
   & \ls{} 2^{\frac{k}2} 2^{\frac{k_2-k_1}2} \la k_1\ra \| \phi \|_{S^+_{k}} \|\psi_1\|_{S^{s_1}_{k_1}} \| \psi_2 \|_{S^{s_2}_{k_2}}.
  \end{align*}
  The factor $2^{\frac{k_2-k_1}2}$ is acceptable in only one instance: when we seek to bound the term by 
$ \| \phi \|_{Z^{+}_{k}}  \| \psi_1 \|_{S^{s_1}_{k_1}} \| \psi_2 \|_{Z^{s_2}_{k_2}}$; precisely we have
\[
| I_{11}(\mathbf{k})| \ls 2^{k(\frac12-2r)+k_1(r-\frac12)} \la k_1 \ra \left( 2^{\frac{k}2+  r (k-k_1+k_2)}  \| \phi \|_{Z^{+}_{k}}  \| \psi_1 \|_{S^{s_1}_{k_1}} \| \psi_2 \|_{Z^{s_2}_{k_2}} \right).
\]
For the other bound, the balance of powers does not work; in that case we can use the $X^{s_1, \frac12 ,1}_{k_1, 4, \frac43}$ information on $\psi_1$, and estimate as follows: 
  \begin{align*}
  | I_{11}(\mathbf{k}) | & \ls{} \sum_{-k_1\preceq j_1 \prec k_1} \Big\{ 2^{-\frac{k_1-j_1}2} 
   \|Q^+_{<j_1} \phi\|_{L^8_t L^\frac83_{x}}   \Big( \sum_{\kappa_1 \in \mathcal{K}_l}\| P_{\kappa_1}  Q_{ j_1}^{s_1}  \psi_1\|_{L^\frac43_t L^4_x}^2 \Big)^{\frac12}
  \cdot \Big( \sum_{\kappa_2 \in \mathcal{K}_l}\| P_{\kappa_2}  Q_{\leq j_1}^{s_2}  \psi_2\|_{L^8_t L^\frac83_{x}}^2 \Big)^{\frac12}\Big\}\\
& \ls{} \sum_{-k_1\preceq j_1 \prec k_1} 2^{-\frac{k_1-j_1}2} 2^{- \frac{j_1}2}   \|Q_{ j_1}^{s_1} \psi_1\|_{X^{s_1, \frac12}_{k_1, 4, \frac43}} 2^{\frac{k}{4}} \|\phi\|_{S^+_{k}}2^{\frac{k_2}{4}} \|\psi_2\|_{S^{s_2}_{k_2}}\\
   & \ls{} 2^{\frac{k}2} 2^{(\rl -\frac12)k_1} \| \phi \|_{S^+_{k}} \|\psi_1\|_{Z_{k_1}} \| \psi_2 \|_{S^{s_2}_{k_2}}.
     \end{align*}
   In this case we obtain
  \[
   | I_{11}(\mathbf{k}) | \ls  2^{(\rl -\frac12-r )k_1} \left( 2^{\frac{k}2+ r (-k+k_1+k_2)} \| \phi \|_{S^{+}_{k}}  \| \psi_1 \|_{Z^{s_1}_{k_1}} \| \psi_2 \|_{Z^{s_2}_{k_2}} \right).
  \]
  In the range $j_1 \succeq k_1$, we forgo the gain from the null-structure and estimate the term in two ways just as above. First we estimate as follows 
\begin{align*}
 | I_{12}(\mathbf{k})\ls{}   &  \sum_{j_1\succeq k_1} \|Q^+_{<j_1} \phi\|_{L^4_{t,x}} 
 \| Q^{s_1}_{j_1} \psi_1\|_{L^2} \| Q_{\leq j_1}^{s_2}  \psi_2\|_{L^4_{t,x}}
 \ls \sum_{j_1\succeq k_1}  2^{-\frac{j_1}{2}} \|\psi_1\|_{S^{s_1}_{k_1}}2^{\frac{k}{2}}\|\phi\|_{S^+_{k}}2^{\frac{k_2}{2}}\|\psi_2\|_{S^{s_2}_{k_2}}\\
    \ls{}& 2^{\frac{k}2} 2^{\frac{k_2-k_1}2} \| \phi \|_{S^+_{k}} \|\psi_1\|_{S^{s_1}_{k_1}} \| \psi_2 \|_{S^{s_2}_{k_2}}.
  \end{align*}
 Just as above, the factor $2^{\frac{k_2-k_1}2}$ is acceptable only when bound the term by $ \| \phi \|_{Z^{+}_{k}}  \| \psi_1 \|_{S^{s_1}_{k_1}} \| \psi_2 \|_{Z^{s_2}_{k_2}}$, and in this case we obtain the same bound as the one we have for $I_{11}(\mathbf{k})$. 

Second we estimate as follows
 \begin{align*}
| I_{12}(\mathbf{k}) & \ls{} \sum_{j_1\succeq k_1} \|Q^+_{<j_1} \phi\|_{L^8_t L^\frac83_{x}}  \| Q^{s_1}_{j_1} \psi_1\|_{L^\frac43_t L^4_x} \| Q_{\leq j_1}^{s_2}  \psi_2\|_{L^8_t L^\frac83_{x}} \ls
2^{\frac{k}2} \| \phi \|_{S^+_{k}} \|\psi_1\|_{X^{s_1, \frac12,1}_{k_1, 4, \frac43}} \| \psi_2 \|_{S^{s_2}_{k_2}} 
   \sum_{ j_1 \succeq k_1} 2^{- \frac{j_1}2}  \\
   & \ls{} 2^{\frac{k}2} 2^{(\rl -\frac12)k_1}  \| \phi \|_{S^+_{k}} \|\psi_1\|_{Z_{k_1}} \| \psi_2 \|_{S^{s_2}_{k_2}}.
 \end{align*}
 Thus in this case we obtain $
 | I_{12}(\mathbf{k}) | \ls 2^{(\rl - \frac12 -r )k_1}  \left( 2^{\frac{k}2+ r (-k+k_1+k_2)} \| \phi \|_{S^{+}_{k}}  \| \psi_1 \|_{Z^{s_1}_{k_1}} \| \psi_2 \|_{Z^{s_2}_{k_2}} \right)$. Collecting all the bounds from the above gives the following:
 \begin{equation} \label{C1I1}
|G_1(\mathbf{k})| \ls \max (\la k_1 \ra 2^{k(\frac12-2r)+k_1(r-\frac12)}, 2^{(\rl - \frac12 -r )k_1}). 
\end{equation}

{\it Contribution of $I_2(\mathbf{k})$:} The analysis here is rather similar to the one for $I_0(\mathbf{k})$. We split
  $I_2(\mathbf{k})=I_{21}(\mathbf{k})+I_{22}(\mathbf{k})$ according to
 $j_2 \prec k_1$ and $j_2\succeq k_1$. Again, by Lemma \ref{lem:mod} there is
  no contribution if $j_2 \prec k_1$ in the case $s_1=+,s_2=-$, whereas in all other choices of signs, we can restrict the sum in $I_{21}(\mathbf{k})$ to $j_2\succeq -k_1$, so that by Lemma \ref{lem:mod} with $2l=k_1+k_2-k-j_2\sim k_1-j_2$ we repeat the argument for $ I_{11}(\mathbf{k})$ (the part of the argument involving only the $S$ type structures), except that we use the $L^4$ type norm for $\psi_1$ and $L^2$ norm for $Q_{j_2}^{s_2} \psi_2$,  to obtain
\[
  I_{21}(\mathbf{k}) \ls{}
\sum_{-k_1\preceq j_2 \prec k_1} 2^{\frac{k}{2}} \|\phi\|_{S^+_{k}} 2^{-\frac{k_1-j_2}2}  2^{\frac{k_1}{2}}  \|\psi_1\|_{S^{s_1}_{k_1}} 2^{-\frac{j_2}{2}}\|\psi_2\|_{S^{s_2}_{k_2}} \ls{} 2^{\frac{k}2}  \la k_1\ra \| \phi \|_{S^+_{k}} \|\psi_1\|_{S^{s_1}_{k_1}} \| \psi_2 \|_{S^{s_2}_{k_2}}.
\]
For the range $j_2 \succeq k_1$, then the same argument as above, but with no gain from the null-structure, gives the bound
\begin{align*}
  I_{22}(\mathbf{k})\ls{}   \sum_{j_2\succeq k_1}   2^{\frac{k}{2}}\|\phi\|_{S^+_{k}}  2^{\frac{k_1}{2}} \|\psi_1\|_{S^{s_1}_{k_1}} 2^{-\frac{j_2}{2}}\|\psi_2\|_{S^{s_2}_{k_2}} \ls{}  2^{\frac{k}2}  \| \phi \|_{S^+_{k}} \|\psi_1\|_{S^{s_1}_{k_1}} \| \psi_2 \|_{S^{s_2}_{k_2}}.
\end{align*}
Thus in this case we obtain
 \begin{equation} \label{C1I2}
|G_2(\mathbf{k})| \ls \la k_1 \ra 2^{-rk_1}.  
\end{equation}

From \eqref{C1I0},  \eqref{C1I1} and  \eqref{C1I2} we obtain, in this case, the following bound
\begin{equation} \label{C1}
|G(\mathbf{k})| \ls \la k_1 \ra 2^{-rk_1} + \max (\la k_1 \ra 2^{k(\frac12-2r)+k_1(r-\frac12)}, 2^{(\rl - \frac12 -r )k_1}).    
\end{equation}

\medskip

  {\bf Case 2:} $|k-k_1 | \leq 10$. A quick inspection of the argument in the previous case shows that we obtain the same bound with the roles of $k_1$ and $k_2$ reversed, that is 
\begin{equation} \label{C2}
|G(\mathbf{k})| \ls \la k_2 \ra 2^{-rk_2} + \max (\la k_2 \ra 2^{k(\frac12-2r)+k_2(r-\frac12)}, 2^{(\rl -\frac12 -r )k_2}).   
\end{equation}

\medskip

  {\bf Case 3:} $|k_1-k_2 | \leq 10$.

 {\it Contribution of $I_0(\mathbf{k})$:} We split
  $I_0(\mathbf{k})=I_{01}(\mathbf{k})+I_{02}(\mathbf{k})$ according to
  $j \prec k$ and $j\succeq k$. Then, due to Lemma \ref{lem:mod} there is
  no contribution if $j \prec k$ in the case $s_1=+,s_2=-$ or $s_1=-, s_2=+$ and $k \prec \min(k_1,k_2)$.
In all remaining cases, we can restrict the sum in $I_{01}$ to $j\succeq -k$, and running the analysis in an identical way to the corresponding term in the Case 1, we obtain $ | I_{01}(\mathbf{k})|  \ls \la k\ra  2^{\frac{k}2} \| \phi \|_{S^+_{k}}  \| \psi_1 \|_{S^{s_1}_{k_1}} \| \psi_2 \|_{S^{s_2}_{k_2}}$.

Let us now consider the range $j \succeq k$. In the case $s_1=s_2$, part iii) of Lemma \ref{lem:mod} implies that the integral is nonzero only if the frequencies in the supports of $\widehat{\psi_1}$ and $\widehat{\psi_2}$ make an angle of at most $2^{k-k_1}$, hence, we choose $l=k_1-k$. In the remaining case where $s_1=-$, $s_2=+$ and $k \succeq \min(k_1,k_2)$ we choose $l=0$.
Again, arguing as for $I_{01}(\mathbf{k})$ (see Case 1) we obtain
\begin{align*}
  |I_{02}(\mathbf{k})| \ls \sum_{ j \succeq k} 2^{-\frac{j}2}
    \| \phi \|_{S_{k}} 2^{-l} 2^{\frac{k_1}2} \|
    \psi_1 \|_{S_{k_1}^{s_1}}
    2^{\frac{k_2}2}  \| \psi_2 \|_{S_{k_2}^{s_2}} \ls 2^{\frac{k}2}
  \| \phi \|_{S^+_{k}}  \| \psi_1 \|_{S^{s_1}_{k_1}} \|\psi_2 \|_{S^{s_2}_{k_2}}.
\end{align*}

In the case $s_1=+$, $s_2=-$ or in the case $s_1=-$, $s_2=+$ and $k \prec \min(k_1,k_2)$, Lemma \ref{lem:mod} implies that there is only a contribution if $j\succeq k_1$. Then, the above argument with $l=0$ provides
\begin{align*}
 | I_{02}(\mathbf{k})| \ls  \sum_{j \succeq k_1}2^{-\frac{j}{2}}
  \| \phi \|_{S^+_{k}}  2^{\frac{k_1}2} \| \psi_1 \|_{S^{s_1}_{k_1}} 2^{\frac{k_2}2}\| \psi_2 \|_{S^{s_2}_{k_2}} \ls{}  2^{\frac{k_2}2}
  \| \phi \|_{S^+_{k}}  \| \psi_1 \|_{S^{s_1}_{k_1}} \|\psi_2 \|_{S^{s_2}_{k_2}}.
\end{align*}

The factor $2^{\frac{k_2}2} = 2^{\frac{k}2} \cdot 2^{\frac{k_2-k}2}$ is fine if we bound the term by 
$ \| \phi \|_{S^{+}_{k}}  \| \psi_1 \|_{Z^{s_1}_{k_1}} \| \psi_2 \|_{Z^{s_2}_{k_2}}$, that is 
\[
 | I_{02}(\mathbf{k})|  \ls 2^{k_1(\frac12-2r)+k(r-\frac12)} \left( 2^{\frac{k}2+ r (-k+k_1+k_2)} \| \phi \|_{S^{+}_{k}}  \| \psi_1 \|_{Z^{s_1}_{k_1}} \| \psi_2 \|_{Z^{s_2}_{k_2}} \right).  
\]
To bound using the other term, we estimate as follows: 
  \begin{align*}
  |I_{02}(\mathbf{k})|  & \ls{}     \sum_{j \succeq k_1}  \|Q^+_{j} \phi\|_{ L^\frac43_t L^4_x }   \|  Q_{ \leq j}^{s_1}  \psi_1\|_{L^8_t L^\frac83_{x}}   \| Q_{\leq j}^{s_2}  \psi_2\|_{L^8_t L^\frac83_{x}} \\
& \ls{}   \sum_{j \succeq k_1} 2^{- \frac{j}2}   \| \phi \|_{X^{s_1,\frac12 ,\infty}_{k_1, 4, \frac43}} 2^{\frac{k_1}{4}} \|\psi_1\|_{S^{s_1}_{k_1}}2^{\frac{k_2}{4}}
\|\psi_2\|_{S^{s_2}_{k_2}}\\
   & \ls{}     2^{\rl k} \|\phi \|_{Z_{k_1}^+} \| \psi_1 \|_{S^{s_1}_{k}} \| \psi_2 \|_{S^{s_2}_{k_2}}.
  \end{align*}
Thus in this case we obtain the bound
\begin{equation} \label{C3I0}
|G_0(\mathbf{k})| \ls \la k \ra 2^{-rk} + \max (2^{k_1(\frac12-2r)+k(r-\frac12)}, 2^{(\rl -r -\frac12)k} ).   
\end{equation}

{\it Contribution of $I_1(\mathbf{k})$:} Again, we split
  $I_1(\mathbf{k})=I_{11}(\mathbf{k})+I_{12}(\mathbf{k})$ according to
  $j_1 \prec k$ and $j_1 \succeq k$. Then, due to Lemma \ref{lem:mod} there is
  no contribution if $j_1 \prec k$ in the case $s_1=+,s_2=-$ or $s_1=-, s_2=+$ and $k \prec \min(k_1,k_2)$.
In all remaining cases, we can restrict the sum in $I_{11}$ to $j_1\succeq -k$, so that by Lemma \ref{lem:mod} with $2l=k_1+k_2-k-j_1$ we have 
 \[
I_{11}(\mathbf{k})= \sum_{-k\preceq j_1<k} \sum_{\kappa_1,\kappa_2 \in \mathcal{K}_l \atop
      \dist(s_1 \kappa_1,s_2\kappa_2)\ls 2^{-l}} \Big\{
\int  Q^+_{<j_1} \phi \cdot \la P_{\kappa_1}Q_{j_1}^{s_1}\ \psi_1, \beta  P_{\kappa_2}\ Q_{\leq j_1}^{s_2}  \psi_2 \ra dx dt\Big\}.
  \] 
  Using Lemma \ref{lem:stable}, H\"older's inequality and Cauchy-Schwarz we obtain
  \begin{align*}
 | I_{11}(\mathbf{k})|  & \ls{} \sum_{-k\preceq j_1 \prec k} \|Q^+_{<j_1} \phi\|_{L^4_{t,x}} 2^{-\frac{k_1+k_2-k-j_1}2} \| Q^{s_1}_{j_1} \psi_1\|_{L^2}  \|  Q_{\leq j_1}^{s_2}  \psi_2\|_{L^4_{t,x}[k_2;l]} \\
& \ls{} \sum_{-k\preceq j_1 \prec k} 2^{\frac{k}{2}} \|\phi\|_{S^+_{k}} 2^{-\frac{k_1+k_2-k-j_1}2}  2^{-\frac{j_1}{2}}\|\psi_1\|_{S^{s_1}_{k_1}}2^{\frac{k_2}{2}} \|\psi_2\|_{S^{s_2}_{k_2}}\\
  &  \ls{} 2^{\frac{k}2} 2^{\frac{k-k_1}{2}}\la k\ra \| \phi \|_{S^+_{k}} \|\psi_1\|_{S^{s_1}_{k_1}} \| \psi_2 \|_{S^{s_2}_{k_2}}.
  \end{align*}
 Let us now consider the case $j_1 \succeq k$. We use a similar dichotomy as for $I_{02}(\mathbf{k})$. In the case $s_1=+$, $s_2=-$ or in the case $s_1=-$, $s_2=+$ and $k \prec \min(k_1,k_2)$, Lemma \ref{lem:mod} implies that there is only a contribution if $j_1\succeq k_2$.
In that case, we obtain from the above argument with $l=0$
\begin{align*}
 | I_{12}(\mathbf{k})| \ls{}   \sum_{j_1\succeq k_2}  2^{\frac{k}{2}}\|\phi\|_{S^+_{k}}  2^{-\frac{j_1}{2}}\|\psi_1\|_{S^{s_1}_{k_1}}2^{\frac{k_2}{2}}\|\psi_2\|_{S^{s_2}_{k_2}}
\ls{}   2^{\frac{k}2} \| \phi \|_{S^+_{k}} \|\psi_1\|_{S^{s_1}_{k_1}} \| \psi_2 \|_{S^{s_2}_{k_2}}.
  \end{align*}

In the case $s_1=s_2$, part iii) of Lemma \ref{lem:mod} implies that the integral is nonzero only if the frequencies in the supports of $\widehat{\psi_1}$ and $\widehat{\psi_2}$ make an angle of at most $2^{k-k_1}$, hence, we choose $l=k_1-k$. Then just as above we obtain
 \begin{align*}
 | I_{12}(\mathbf{k})|  & \ls{} \sum_{j_1\succeq k} \|Q^+_{<j_1} \phi\|_{L^4_{t,x}} 2^{-(k_1-k)} \| Q^{s_1}_{j_1} \psi_1\|_{L^2}  \|  Q_{\leq j_1}^{s_2}  \psi_2\|_{L^4_{t,x}[k_2;l]} \\
& \ls{} \sum_{ j_1 \succeq k} 2^{\frac{k}{2}} \|\phi\|_{S^+_{k}} 2^{-(k_1-k)}  2^{-\frac{j_1}{2}}\|\psi_1\|_{S^{s_1}_{k_1}}2^{\frac{k_2}{2}} \|\psi_2\|_{S^{s_2}_{k_2}}\\
  &  \ls{} 2^{\frac{k}2} 2^{\frac{k-k_1}2}\| \phi \|_{S^+_{k}} \|\psi_1\|_{S^{s_1}_{k_1}} \| \psi_2 \|_{S^{s_2}_{k_2}}.
  \end{align*}
In the remaining case where $s_1=-$, $s_2=+$ and $k \succeq \min(k_1,k_2)$ we choose $l=0$ and, just as before, we obtain
 we obtain
\begin{align*}
 | I_{12}(\mathbf{k})| \ls{}   \sum_{j_1\geq k}  2^{\frac{k}{2}} \|\phi\|_{S^+_{k}}   2^{-\frac{j_1}{2}}\|\psi_1\|_{S^{s_1}_{k_1}}2^{\frac{k_2}{2}}\|\psi_2\|_{S^{s_2}_{k_2}}
\ls{}   2^{\frac{k}2}  \| \phi \|_{S^+_{k}} \|\psi_1\|_{S^{s_1}_{k_1}} \| \psi_2 \|_{S^{s_2}_{k_2}}.
  \end{align*}

Thus in this case we obtain the bound
\begin{equation} \label{C3I1}
|G_1(\mathbf{k})| \ls \la k \ra 2^{-rk}.   
\end{equation}

{\it Contribution of $I_2(\mathbf{k})$:} This is treated in the same way as $I_1(\mathbf{k})$.

Thus in this case we obtain the bound
\begin{equation} \label{C3I}
|G(\mathbf{k})| \ls \la k \ra 2^{-rk} + \max (2^{k_1(\frac12-2r)+k(r-\frac12)}, 2^{(\rl-r -\frac12)k}).   
\end{equation}

A quick investigation of the bounds in \eqref{C1}, \eqref{C2} and \eqref{C3I} shows that \eqref{G} holds true if $r \geq \frac12$ and $r +\frac12 \geq \rl$.

\end{proof}
 
\subsection{Estimates for dyadic pieces - part 2}
Our main result in this section is the following.

\begin{pro} \label{PTR2} Let $s_1,s_2\in\{+,-\}$. There exists 
  functions $\tilde G^1, \tilde G^2 : \N^3_0 \rightarrow (0,\infty)$ such that for all  $\phi=P_{k}\phi,\psi_i=P_{k_i}\Pi_{s_i}(D)\psi_i$, $i=1,2$, the  following estimate holds true:
  \begin{equation} \label{cunn22} 
  \Big| \int Q_{\succeq -k}^+ \phi \la \psi_1, \beta  \psi_2 \ra dx dt\Big| \ls{} \tilde G^1(\mathbf{k}) 2^{k(-r+\frac12+\rl)} 2^{r(k_1+k_2)}
     \| Q_{\succeq -k}^+ \phi\|_{ X^{+, \frac12, \infty}_{k,4,\frac43}}\|\psi_1\|_{S^{s_1}_{k_1}}\|\psi_2\|_{S^{s_2}_{k_2}},
\end{equation}
and 
 \begin{equation} \label{cunn222} 
  \Big| \int \phi \la Q_{\succeq -k_1} ^{s_1} \psi_1, \beta  \psi_2 \ra dx dt\Big| \ls{} \tilde G^2(\mathbf{k})  2^{k_1(-r+\rl)} 2^{(r+\frac12)k+r k_2}
    \|\phi\|_{S_k^{+}}  \| Q^{s_1}_{\succeq -k_1} \psi_1\|_{ X^{s_1,\frac12, \infty}_{k_1,4,\frac43}}   \|\psi_2\|_{S_{k_2}^{s_2}}   .
\end{equation}
Assuming that $r \geq \frac12$ and $r_0 \geq 1$, the functions $\tilde G^1, \tilde G^2$ satisfy \eqref{G}. 
\end{pro}

It is clear that \eqref{eq:red-est-2}, with the choice $r_0=1$ and $r > \frac12$, is implied by the above result. 

\begin{proof}[Proof of \eqref{cunn222}] We fix $j_1 \succeq -k_1$ and let 
\[
I_{j_1}(\mathbf{k}) = \int \phi \la Q_{j_1} ^{s_1} \psi_1, \beta  \psi_2 \ra dx dt. 
\]
The estimate in \eqref{cunn222} will be retrieved by summing with respect to $j_1$ the estimates on $I_{j_1}(\mathbf{k})$.  We decompose $ I_{j_1}(\mathbf{k})=I_{0}(\mathbf{k})+I_{1}(\mathbf{k})+I_{2}(\mathbf{k})$, where
  \begin{align*}
    I_0(\mathbf{k}):=&  \int Q_{\leq j_1}^+ \phi \, \la Q^{s_1}_{j_1}\psi_1, \beta Q^{s_2}_{\leq j_1}\psi_2 \ra dx dt,\ \ 
    I_1(\mathbf{k}):=\sum_{j > j_1} \int Q_{j}^+ \phi  \,\la Q^{s_1}_{j_1} \psi_1, \beta Q^{s_2}_{\leq j}\psi_2 \ra dxdt,\\
    I_2(\mathbf{k}):=&\sum_{j_2 > j_1}\int Q_{< j_2}^+ \phi  \,\la   Q^{s_1}_{j_1}\psi_1, \beta Q^{s_2}_{j_2}\psi_2 \ra dx dt. \\
  \end{align*}
  In the above we abuse a little notation since we do not bring the factor $j_1$ in the notation for $ I_0(\mathbf{k}),  I_1(\mathbf{k}),  I_2(\mathbf{k})$; this is done in order to have simpler notation later.

  We split the argument into three cases.

\medskip

  {\bf Case 1:} $|k-k_1 | \leq 10$.

{\it Contribution of $I_0(\mathbf{k})$:}  
We split
  $I_0(\mathbf{k})=I_{01}(\mathbf{k})+I_{02}(\mathbf{k})$ according to
 $j_1 \prec k_2$ and $j_1\succeq k_2$. By Lemma \ref{lem:mod} there is
  no contribution if $j_1 \prec k_1$ in the case $s_1=+,s_2=-$.  With all other choices of signs, we can restrict the sum in $I_{01}$ to $j_1\succeq -k_2$, so that by Lemma \ref{lem:mod} with $2l=k_1+k_2-k-j_1\sim k_2-j_1$  we have
\begin{align*}
I_{01}(\mathbf{k})= \sum_{\kappa_1,\kappa_2 \in \mathcal{K}_l \atop
      \dist(s_1 \kappa_1,s_2\kappa_2)\ls 2^{-l}} 
      \Big\{\int  Q^+_{\leq j_1} \phi \cdot \la P_{\kappa_1}Q_{j_1}^{s_1} \psi_1, \beta  P_{\kappa_2} Q_{\leq j_1}^{s_2}  \psi_2 \ra dx dt\Big\}.
  \end{align*}
 A cap $\kappa_1 \in  \mathcal{K}_l$ contains $\approx 2^{\frac{k_1-k_2}2}$ smaller caps at angular scale $\frac{k_1-j_1}2$. From \eqref{Kllg} we obtain
  \[
  \left( \sum_{\kappa_1 \in \mathcal{K}_l} \| P_{\kappa_1}Q_{j_1}^{s_1} \psi_1 \|^2_{L^4_t L^\frac43_x} \right)^\frac12 \ls 2^{\frac{k_1-k_2}4} 2^{-\frac{j_1}2} \| Q_{j_1}^{s_1} \psi_1 \|_{X^{s_1, \frac12}_{k_1,4,\frac43}}.
  \]
 Using this, we estimate as follows:
 \[
  \begin{split}
|  I_{01}(\mathbf{k})| \ls &  \|Q_{\leq j_1}^+ \phi  \|_{L^\frac83_t L^8_x} 
\sum_{\kappa_1,\kappa_2 \in \mathcal{K}_l \atop    \dist(s_1 \kappa_1,s_2\kappa_2)\ls 2^{-l}} 
   2^{-l}   \| Q_{ j_1}^{s_1}  P_{\kappa_1}\psi_1\|_{L^4_t L^\frac43_x} \|Q_{\leq j_1}^{s_2} P_{\kappa_2}\psi_2\|_{L^\frac83_t L^8_x } \\
  \ls &   \|Q_{\leq j_1}^+ \phi  \|_{L^\frac83_t L^8_x} 2^{-l} 2^{-\frac{j_1}2} 2^{\frac{k_1-k_2}4} \| Q_{ j_1}^{s_1} \psi_1 \|_{X^{s_1,\frac12}_{k_1,4,\frac43}}  \| Q_{\leq j_1}^{s_2} \psi_2 \|_{L^\frac83_t L^8_x[k_2;l]} \\
  \ls &  2^{k} \| \phi  \|_{S_k^+} 2^{-\frac{3k_2}4}  \| Q_{ j_1}^{s_1} \psi_1 \|_{X^{s_1, \frac12}_{k_1, 4,\frac43}}  2^{\frac{3k_2}4} \| \psi_2 \|_{S^{s_2}_{k_2}} 
   =   2^{k} \| \phi  \|_{S_k^+}  \| Q_{ j_1}^{s_1} \psi_1 \|_{X^{s_1,\frac12}_{k_1,4,\frac43}}  \| \psi_2 \|_{S^{s_2}_{k_2}}.
\end{split}
  \]
  The case $j_1 \succeq k_2$ is similar, but easier. We work in a similar manner as above, but with $l=0$, and, based on \eqref{Kllg}, record the estimate
 \[
  \| Q_{j_1}^{s_1} \psi_1 \|_{L^4_t L^\frac43_x}  \ls \la 2^{\frac{k_1-j_1}4} \ra 2^{-\frac{j_1}2}  \| Q_{j_1}^{s_1} \psi_1 \|_{X^{s_1,\frac12}_{k_1,4,\frac43}}.
  \]
From this we obtain
\[
  \begin{split}
|  I_{02}(\mathbf{k})| \ls &  \|Q_{\leq j_1}^+ \phi  \|_{L^\frac83_t L^8_x} 
     \| Q_{ j_1}^{s_1}  \psi_1\|_{L^4_t L^\frac43_x} \|Q_{\leq j_1}^{s_2} \psi_2\|_{L^\frac83_t L^8_x } \\
  \ls &   2^{\frac{3k}4} \| \phi  \|_{S_k^+} 2^{-\frac{j_1}2} \la 2^{\frac{k_1-j_1}4} \ra \| Q_{j_1}^{s_1}  \psi_1 \|_{X^{s_1,\frac12}_{k_1,4,\frac43}} 2^{\frac{3k_2}4} \| \psi_2 \|_{S^{s_2}_{k_2}}  \\
 \ls  &  2^{k}  2^{\frac{k_2-j_1}2} \| \phi  \|_{S_k^+}  \| Q_{j_1}^{s_1}  \psi_1 \|_{X^{s_1, \frac12}_{k_1,4,\frac43}}   \| \psi_2 \|_{S^{s_2}_{k_2}} .
\end{split}
\]
Thus we conclude with the following
\begin{equation} \label{I0kcc1}
|  I_{0}(\mathbf{k})| \ls 2^k \min(1, 2^{\frac{k_2-j_1}2} ) \| \phi  \|_{S_k^+}  \| Q_{j_1}^{s_1}  \psi_1 \|_{X^{s_1, \frac12}_{k_1,4,\frac43}}   \| \psi_2 \|_{S^{s_2}_{k_2}} .
\end{equation}

{\it Contribution of $I_1(\mathbf{k})$:}  
We split  $I_1(\mathbf{k})=I_{11}(\mathbf{k})+I_{12}(\mathbf{k})$ according to
 $j \prec k_2$ and $j \succeq k_2$. Moreover, we let $I^j_{11}(\mathbf{k}), I^j_{12}(\mathbf{k})$ be the corresponding term with the value of $j$ fixed in each expression.

 Again, by Lemma \ref{lem:mod} there is  no contribution if $j \prec k_1$ in the case $s_1=+,s_2=-$.  With all other choices of signs, we can restrict the sum in $I_{11}$ to $j\succeq -k_2$, so that by Lemma \ref{lem:mod} with $2l=k_1+k_2-k-j\sim k_2-j$ we have
\begin{align*}
I^j_{11}(\mathbf{k})= \sum_{\kappa_1,\kappa_2 \in \mathcal{K}_l \atop
      \dist(s_1 \kappa_1,s_2\kappa_2)\ls 2^{-l}} 
      \Big\{\int  Q^+_{j} \phi \cdot \la P_{\kappa_1}Q_{j_1}^{s_1} \psi_1, \beta  P_{\kappa_2} Q_{\leq j}^{s_2}  \psi_2 \ra dx dt\Big\}.
  \end{align*}
  Based on the size of the Fourier support of the underlying function in each case we obtain
  \[
 \| Q_j^+ \phi\|_{L^2_t L^4_x} \ls 2^{\frac{k}2} \| Q_j^+ \phi\|_{L^2_{t,x}} \ls 2^{\frac{k}2} 2^{-\frac{j}2} \| Q_j^+ \phi\|_{S_{k}^+}, \ \|P_{\kappa_2} Q_{\leq j}^{s_2} \psi_2 \|_{L^4_t L^\infty_x} \ls 2^{\frac{3k_2+j}8}  \|P_{\kappa_2} Q_{\leq j}^{s_2} \psi_2 \|_{L^4_t L^4_x}.
 \]
   From the later estimate it follows that
 \begin{equation} \label{aux1}
 \left( \sum_{\kappa_2 \in \mathcal{K}_l}  \|P_{\kappa_2} Q_{\leq j}^{s_2} \psi_2 \|_{L^4_t L^\infty_x}^2 \right)^\frac12 \ls 2^{\frac{3k_2+j}8}  2^{\frac{k_2}2} \| \psi_2 \|_{S^{s_2}_{k_2}}.
 \end{equation}
  A cap $\kappa_1$ contains $\approx  \frac{2^{\frac{k_1-j_1}2}}{2^{\frac{k_2-j}2}} = 2^{\frac{k_1-k_2}2} 2^{\frac{j-j_1}2}$ caps $\kappa' \in \mathcal{K}_{\frac{k_1-j_1}2}$; invoking \eqref{Kllg} we obtain
\begin{equation} \label{psi1bg}
\left( \sum_{\kappa_1 \in \mathcal{K}_{l}}  \| P_{\kappa_1} Q^{s_1}_{j_1} \psi_1 \|_{L^4_t L^\frac43_x} ^2 \right)^\frac12
\ls 2^{\frac{k_1-k_2}4} 2^{\frac{j-j_1}4} 2^{-\frac{j_1}2} \|Q_{j_1}^{s_1} \psi_1\|_{X_{k_1,4,\frac32}^{s_1,\frac12}}. 
\end{equation}
Using the above estimates we bound as follows:
 \[
  \begin{split}
|  I_{11}^j(\mathbf{k})| \ls &  \|Q_{j}^+ \phi  \|_{L^2_t L^4_x} 
\sum_{\kappa_1,\kappa_2 \in \mathcal{K}_l \atop    \dist(s_1 \kappa_1,s_2\kappa_2)\ls 2^{-l}} 
   2^{-l}   \| Q_{ j_1}^{s_1}  P_{\kappa_1}\psi_1\|_{L^4_t L^\frac43_x} \|Q_{\leq j}^{s_2} P_{\kappa_2}\psi_2\|_{L^4_t L^\infty_x } \\
  \ls &   2^{-l} 2^{-\frac{j}2} 2^{\frac{k_1}2} 2^{\frac{3k_2+j}8}  2^{\frac{k_2}2} 2^{\frac{k_1-k_2}4} 2^{\frac{j-j_1}4} 2^{-\frac{j_1}2}  \| Q_j^+ \phi\|_{S_{k}^+}
  \|Q_{j_1}^{s_1} \psi_1\|_{X_{k_1,4,\frac43}^{s_1,\frac12}} \| \psi_2 \|_{S^{s_2}_{k_2}} \\
  \approx  &    2^{\frac{3k_2+3j}8} 2^{\frac{3k_1-k_2}4} 2^{-\frac{3j_1}4}  \| Q_j^+ \phi\|_{S_{k}^+}
  \|Q_{j_1}^{s_1} \psi_1\|_{X_{k_1,4,\frac43}^{s_1,\frac12}} \| \psi_2 \|_{S^{s_2}_{k_2}}.
\end{split}
  \]
From this it follows that
\[
|  I_{11}(\mathbf{k})| \ls \sum_{-k_2 \preceq j \prec k_2} |  I_{11}^j(\mathbf{k})|
\ls 2^{\frac{3k_1+2k_2}4} 2^{-\frac{3j_1}4}  \| \phi\|_{S_{k}^+}
  \|Q_{j_1}^{s_1} \psi_1\|_{X_{k_1,4,\frac32}^{s_1,\frac12}} \| \psi_2 \|_{S^{s_2}_{k_2}}.
\]
We now turn our attention to the case $j \succeq k_2$. The argument is fairly similar, but simpler. First we do not use any cap localization. Then we replace \eqref{aux1} and \eqref{psi1bg} with 
\[
 \|Q_{\leq j}^{s_2} \psi_2 \|_{L^4_t L^\infty_x} \ls 2^{\frac{k_2}2}  2^{\frac{k_2}2} \| \psi_2 \|_{S^{s_2}_{k_2}}, \ \ 
\| Q^{s_1}_{j_1} \psi_1 \|_{L^4_t L^\frac43_x}
\ls \la 2^{\frac{k_1-j_1}4} \ra 2^{-\frac{j_1}2} \|Q_{j_1}^{s_1} \psi_1\|_{X_{k_1,4,\frac32}^{s_1,\frac12}}. 
\]
Using the above estimates we bound as follows:
 \[
  \begin{split}
|  I_{12}^j(\mathbf{k})| \ls &  \|Q_{j}^+ \phi  \|_{L^2_t L^4_x} 
\| Q_{ j_1}^{s_1} \psi_1\|_{L^4_t L^\frac43_x} \|Q_{\leq j}^{s_2} \psi_2\|_{L^4_t L^\infty_x } \\
  \ls &   2^{-\frac{j}2} 2^{\frac{k}2} 2^{k_2} \la 2^{\frac{k_1-j_1}4} \ra 2^{-\frac{j_1}2}  \| Q_j^+ \phi\|_{S_{k}^+}
  \|Q_{j_1}^{s_1} \psi_1\|_{X_{k_1,4,\frac32}^{s_1,\frac12}} \| \psi_2 \|_{S^{s_2}_{k_2}} \\
  \approx  &    2^{\frac{k_2-j}2} 2^{\frac{k_1+k_2-j_1}2}  \la 2^{\frac{k_1-j_1}4} \ra \| Q_j^+ \phi\|_{S_{k}^+}
  \|Q_{j_1}^{s_1} \psi_1\|_{X_{k_1,4,\frac32}^{s_1,\frac12}} \| \psi_2 \|_{S^{s_2}_{k_2}}.
\end{split}
  \]
From this it follows that
\[
|  I_{12}(\mathbf{k})| \ls \sum_{j \succeq k_2} |  I_{12}^j(\mathbf{k})|
\ls 2^{\frac{k_1+k_2-j_1}2}  \la 2^{\frac{k_1-j_1}4} \ra \| \phi\|_{S_{k}^+}
  \|Q_{j_1}^{s_1} \psi_1\|_{X_{k_1,4,\frac32}^{s_1,\frac12}} \| \psi_2 \|_{S^{s_2}_{k_2}}.
\]
Adding the estimates for $I_{11}(\mathbf{k})$ and $ I_{12}(\mathbf{k})$, gives
the following estimate
\begin{equation} \label{I1kcc1}
    |  I_{1}(\mathbf{k})| \ls 2^{\frac{k_1+k_2-j_1}2}  \la 2^{\frac{k_1-j_1}4} \ra \| \phi\|_{S_{k}^+}  \|Q_{j_1}^{s_1} \psi_1\|_{X_{k_1,4,\frac32}^{s_1,\frac12}} \| \psi_2 \|_{S^{s_2}_{k_2}}.
\end{equation}

{\it Contribution of $I_2(\mathbf{k})$:}  We split  $I_2(\mathbf{k})=I_{21}(\mathbf{k})+I_{22}(\mathbf{k})$ according to
 $j_2 \prec k_2$ and $j_2 \succeq k_2$. Moreover, we let $I^{j_2}_{11}(\mathbf{k}), I^{j_2}_{12}(\mathbf{k})$ be the corresponding term with the value of $j_2$ fixed.

 Again, by Lemma \ref{lem:mod} there is
  no contribution if $j_2 \prec k_1$ in the case $s_1=+,s_2=-$.  With all other choices of signs, we can restrict the sum in $I_{21}^{j_2}$ to $j_2 \succeq -k_2$, so that by Lemma \ref{lem:mod} with $2l=k_1+k_2-k-j\sim k_2-j_2$ we have
\begin{align*}
I^{j_2}_{21}(\mathbf{k})= \sum_{\kappa_1,\kappa_2 \in \mathcal{K}_l \atop
      \dist(s_1 \kappa_1,s_2\kappa_2)\ls 2^{-l}} 
      \Big\{\int  Q^+_{< j_2} \phi \cdot \la P_{\kappa_1}Q_{j_1}^{s_1} \psi_1, \beta  P_{\kappa_2} Q_{j_2}^{s_2}  \psi_2 \ra dx dt\Big\}.
  \end{align*}
  Based on the size of the Fourier support we obtain
$ \|P_{\kappa_2} Q_{j_2}^{s_2} \psi_2 \|_{L^2_t L^\infty_x} \ls 2^{\frac{3k_2+j_2}4}  \|P_{\kappa_2} Q_{j_2}^{s_2} \psi_2 \|_{L^2_{t,x}}$, from which it follows that 
\[
 \left( \sum_{\kappa_2 \in \mathcal{K}_l} \|P_{\kappa_2} Q_{j_2}^{s_2} \psi_2 \|_{L^2_t L^\infty_x}^2 \right)^\frac12 \ls 2^{\frac{3k_2+j_2}4} 2^{-\frac{j_2}2} \|P_{\kappa_2} Q_{j_2}^{s_2} \psi_2 \|_{X^{s_2, \frac12}}.
 \]

 We have established in \eqref{psi1bg} (with the roles of $j$ replaced by $j_2$ here) the following
 \[
\left( \sum_{\kappa_1 \in \mathcal{K}_{l}}  \| P_{\kappa_1} Q^{s_1}_{j_1} \psi_1 \|_{L^4_t L^\frac43_x}^2 \right)^\frac12
\ls 2^{\frac{k_1-k_2}4} 2^{\frac{j_2-j_1}4} 2^{-\frac{j_1}2} \|Q_{j_1}^{s_1} \psi_1\|_{X_{k_1,4,\frac43}^{s_1,\frac12}}. 
\]
Based on the above estimates we obtain:
  \[
  \begin{split}
  |I_{21}^{j_2}(\mathbf{k})|  & \ls  \| Q^+_{< j_2} \phi \|_{L^4_{t,x}} \cdot 
  \sum_{\kappa_1,\kappa_2 \in \mathcal{K}_l \atop
      \dist(s_1 \kappa_1,s_2\kappa_2)\ls 2^{-l}}
   2^{-l} \| P_{\kappa_1}Q_{j_1}^{s_1} \psi_1 \|_{L^4_t L^\frac43_x} 
    \| P_{\kappa_2} Q_{j_2}^{s_2}  \psi_2 \|_{L^2_t L^\infty_x}
  \\
  & \ls 2^{-l} 2^{-\frac{j_2}2} 2^{\frac{k_1}2} 2^{\frac{3k_2+j_2}4}  2^{\frac{k_1-k_2}4} 2^{\frac{j_2-j_1}4} 2^{-\frac{j_1}2}  \| \phi\|_{S_{k}^+}
  \|Q_{j_1}^{s_1} \psi_1\|_{X_{k_1,4,\frac43}^{s_1,\frac12}} \| Q_{ j_2}^{s_2} \psi_2 \|_{S^{s_2}_{k_2}} \\
  & \approx    2^{\frac{k_2+2j_2}4} 2^{\frac{3k_1-k_2}4} 2^{-\frac{3j_1}4}  \| \phi\|_{S_{k}^+}
  \|Q_{j_1}^{s_1} \psi_1\|_{X_{k_1,4,\frac43}^{s_1,\frac12}} \| Q_{ j_2}^{s_2} \psi_2 \|_{S^{s_2}_{k_2}}.
  \end{split}
  \] 
From this it follows that
\[
 |I_{21}(\mathbf{k})|  \ls  \sum_{-k_2 \preceq j_2 \prec k_2} |I_{21}^{j_2}(\mathbf{k})| \ls 2^{\frac{3k_1+2k_2}4} 2^{-\frac{3j_1}4}  \| \phi\|_{S_{k}^+}
  \|Q_{j_1}^{s_1} \psi_1\|_{X_{k_1, 4,\frac43}^{s_1,\frac12}} \| \psi_2 \|_{S^{s_2}_{k_2}}.
\]
When $j_2 \succeq k_2$, we do not use a cap localization and replace the above basic estimates with
\[
 \|Q_{j_2}^{s_2} \psi_2 \|_{L^2_t L^\infty_x} \ls 2^{k_2} 2^{-\frac{j_2}2}  \| Q_{j_2}^{s_2} \psi_2 \|_{X^{s_2,\frac12}},
\  \| Q^{s_1}_{j_1} \psi_1 \|_{L^4_t L^\frac43_x}
\ls \la 2^{\frac{k_1-j_1}4} \ra 2^{-\frac{j_1}2} \|Q_{j_1}^{s_1} \psi_1\|_{X_{k_1, 4,\frac43}^{s_1,\frac12}}. 
\]
Based on the above estimates we obtain:
\[
  \begin{split}
  |I_{21}^{j_2}(\mathbf{k})|  & \ls  \| Q^+_{< j_2} \phi \|_{L^4_{t,x}}
    \| Q_{j_1}^{s_1} \psi_1 \|_{L^4_t L^\frac43_x} 
    \| Q_{j_2}^{s_2}  \psi_2 \|_{L^2_t L^\infty_x}
  \\
  & \ls 2^{\frac{k}2} 2^{k_2} 2^{-\frac{j_2}2}  \la 2^{\frac{k_1-j_1}4} \ra 
  2^{-\frac{j_1}2}  \| \phi\|_{S_{k}^+}
  \|Q_{j_1}^{s_1} \psi_1\|_{X_{k_1,4,\frac43}^{s_1,\frac12}} \| Q_{ j_2}^{s_2} \psi_2 \|_{S^{s_2}_{k_2}} \\
  & \approx    2^{\frac{k_2-j_2}2} 2^{\frac{k_1+k_2-j_1}2}  \la 2^{\frac{k_1-j_1}4} \ra \| \phi\|_{S_{k}^+}
  \|Q_{j_1}^{s_1} \psi_1\|_{X_{k_1,4,\frac43}^{s_1,\frac12}} \| Q_{ j_2}^{s_2} \psi_2 \|_{S^{s_2}_{k_2}}.
  \end{split}
  \] 
From this it follows that
\[
 |I_{22}(\mathbf{k})|  \ls  \sum_{j_2 \succeq k_2} |I_{22}^{j_2}(\mathbf{k})| \ls 2^{\frac{k_1+k_2-j_1}2}  \la 2^{\frac{k_1-j_1}4} \ra \| \phi\|_{S_{k}^+}
  \|Q_{j_1}^{s_1} \psi_1\|_{X_{k_1,4,\frac32}^{s_1,\frac12}} \| \psi_2 \|_{S^{s_2}_{k_2}}.
\]
Adding the estimates for $I_{21}(\mathbf{k})$ and $ I_{22}(\mathbf{k})$, gives
the following estimate
\begin{equation} \label{I2kcc1}
    |I_{2}(\mathbf{k})|  \ls 2^{\frac{k_1+k_2-j_1}2}  \la 2^{\frac{k_1-j_1}4} \ra  \| \phi\|_{S_{k}^+}  \|Q_{j_1}^{s_1} \psi_1\|_{X_{k_1,4,\frac43}^{s_1,\frac12}} \| \psi_2 \|_{S^{s_2}_{k_2}}.
\end{equation}
Adding the estimates \eqref{I0kcc1}, \eqref{I1kcc1} and \eqref{I2kcc1} and performing the summation with respect to $j_1 \succeq -k_1$ gives the following estimate
\begin{equation}
\Big| \int \phi \la Q_{\succeq -k_1} ^{s_1} \psi_1, \beta  \psi_2 \ra dx dt\Big| \ls{}  2^{\frac{3k_1+k_2}2}  \| \phi\|_{S_{k}^+}  \|\psi_1\|_{X_{k_1,4,\frac43}^{s_1,\frac12,\infty}} \| \psi_2 \|_{S^{s_2}_{k_2}},
\end{equation}
from which we conclude with 
\begin{equation} \label{tG2cc1}
\tilde G^2(\mathbf{k}) = 2^{\frac{k_2}2} 2^{\frac{3k_1}2} 2^{k_1(r-\rl)} 2^{-(r+\frac12)k -rk_2} \approx 2^{k_2(\frac12-r)} 2^{k_1(1-\rl)}
\end{equation}
in this case. 
  
\medskip

  {\bf Case 2:} $|k-k_2 | \leq 10$.

{\it Contribution of $I_0(\mathbf{k})$:}  We split
  $I_0(\mathbf{k})=I_{01}(\mathbf{k})+I_{02}(\mathbf{k})$ according to
 $j_1 \prec k_1$ and $j_1\succeq k_1$. Again, by Lemma \ref{lem:mod} there is
  no contribution if $j_1 \prec k_2$ in the case $s_1=+,s_2=-$.  With all other choices of signs, we can restrict the sum in $I_{01}$ to $j_1\succeq -k_1$, so that by Lemma \ref{lem:mod} with $2l=k_1+k_2-k-j_1\sim k_1-j_1$ we have
\begin{align*}
I_{01}(\mathbf{k})=  \sum_{\kappa_1,\kappa_2 \in \mathcal{K}_l \atop
      \dist(s_1 \kappa_1,s_2\kappa_2)\ls 2^{-l}} 
      \Big\{\int  Q^+_{\leq j_1} \phi \cdot \la P_{\kappa_1}Q_{j_1}^{s_1} \psi_1, \beta  P_{\kappa_2} Q_{\leq j_1}^{s_2}  \psi_2 \ra dx dt\Big\},
  \end{align*}
and estimate as follows:
 \[
  \begin{split}
|  I_{01}(\mathbf{k})| \ls &  \|Q_{\leq j_1}^+ \phi  \|_{L^\frac83_t L^8_x} 
\sum_{\kappa_1,\kappa_2 \in \mathcal{K}_l \atop    \dist(s_1 \kappa_1,s_2\kappa_2)\ls 2^{-l}} 
   2^{-l}   \| Q_{ j_1}^{s_1}  P_{\kappa_1}\psi_1\|_{L^4_t L^\frac43_x} \|Q_{\leq j_1}^{s_2} P_{\kappa_2}\psi_2\|_{L^\frac83_t L^8_x } \\
  \ls &   \|Q_{\leq j_1}^+ \phi  \|_{L^\frac83_t L^8_x} 2^{-l} 2^{-\frac{j_1}2}  \| Q_{j_1}^{s_1} \psi_1 \|_{X^{s_1, \frac12}_{k_1,4,\frac43}}  \| \psi_2 \|_{L^\frac83_t L^8_x[k_2;l]} \\
  \ls &  2^{\frac{3k}4} \| \phi  \|_{S_k^+} 2^{-\frac{k_1}2}  \| Q_{j_1}^{s_1} \psi_1 \|_{X^{s_1, \frac12}_{k_1,4,\frac43}}  2^{\frac{3k_2}4} \| \psi_2 \|_{S^{s_2}_{k_2}}. 
\end{split}
  \]
The case $j_1 \succeq k_1$ is similar, but with no cap localization. Here we obtain
\[
|  I_{02}(\mathbf{k})| \ls 2^{\frac{3k-k_1}2} 2^{\frac{k_1-j_1}2} \| \phi  \|_{S_k^+}   \| Q_{j_1}^{s_1} \psi_1 \|_{X^{s_1,\frac12}_{k_1,4,\frac43}}  \| \psi_2 \|_{S^{s_2}_{k_2}}. 
\]
Thus we conclude with 
\begin{equation} \label{I0kb}
|  I_{0}(\mathbf{k})| \ls 2^{\frac{3k-k_1}2} \min(1,2^{\frac{k_1-j_1}2}) \| \phi  \|_{S_k^+}   \| Q_{j_1}^{s_1} \psi_1 \|_{X^{s_1, \frac12}_{k_1,4,\frac43}}  \| \psi_2 \|_{S^{s_2}_{k_2}}. 
\end{equation}

{\it Contribution of $I_1(\mathbf{k})$:}  
We split
  $I_1(\mathbf{k})=I_{11}(\mathbf{k})+I_{12}(\mathbf{k})$ according to
 $j \prec k_1$ and $j \succeq k_1$. Again, by Lemma \ref{lem:mod} there is
  no contribution if $j \prec k_1$ in the case $s_1=+,s_2=-$.  With all other choices of signs, we can restrict the sum in $I_{11}$ to $j\succeq -k_1$, so that by Lemma \ref{lem:mod} with $2l=k_1+k_2-k-j\sim k_1-j$ we have
\begin{align*}
I^j_{11}(\mathbf{k})= \sum_{\kappa_1,\kappa_2 \in \mathcal{K}_l \atop
      \dist(s_1 \kappa_1,s_2\kappa_2)\ls 2^{-l}} 
      \Big\{\int  Q^+_{j} \phi \cdot \la P_{\kappa_1}Q_{j_1}^{s_1} \psi_1, \beta  P_{\kappa_2} Q_{\leq j}^{s_2}  \psi_2 \ra dx dt\Big\}.
  \end{align*}
For the next estimate we use two ingredients: a cap in  $\mathcal{K}_l$ contains $\approx 2^{\frac{j-j_1}2}$ caps in $\mathcal{K}_{l'}, l'=\frac{k_1-j_1}2$ and 
  the Bernstein inequality for caps $k' \in \mathcal{K}_{l'}$  which gives us  $\| Q_{ j_1}^{s_1}  P_{\kappa'}\psi_1\|_{L^4_{t,x}} \ls 2^{\frac{k_1+\frac{k_1+j_1}2}2} \| Q_{ j_1}^{s_1}  P_{\kappa'}\psi_1\|_{L^4_{t} L^\frac43_x}$;  thus we invoke \eqref{Kllg} to obtain
    \[
  \left( \sum_{\kappa_1 \in \mathcal{K}_l}  \| Q_{ j_1}^{s_1}  P_{\kappa_1}\psi_1\|_{L^4_{t,x}}^2 \right)^\frac12 \ls 
  2^{\frac{j-j_1}4} 2^{\frac{k_1+\frac{k_1+j_1}2}2} 2^{-\frac{j_1}2}\| Q_{ j_1}^{s_1}  \psi_1 \|_{X^{s_1, \frac12}_{k_1,4,\frac43}}. 
  \]
From this we obtain:
 \[
  \begin{split}
|  I^j_{11}(\mathbf{k})| \ls &  \|Q_{ j}^+ \phi  \|_{L^2_{t,x}} 
\sum_{\kappa_1,\kappa_2 \in \mathcal{K}_l \atop    \dist(s_1 \kappa_1,s_2\kappa_2)\ls 2^{-l}} 
   2^{-l}   \| Q_{ j_1}^{s_1}  P_{\kappa_1}\psi_1\|_{L^4_{t,x}} \|Q_{\leq j}^{s_2} P_{\kappa_2}\psi_2\|_{L^4_{t,x} } \\
  \ls &   2^{-\frac{j}2} \|  \phi  \|_{S_k^+} 2^{-l} 2^{-\frac{j_1}2} 2^{\frac{j-j_1}4}  2^{\frac{k_1+\frac{k_1+j_1}2}2}  \| Q_{ j_1}^{s_1} \psi_1 \|_{X^{s_1, \frac12}_{k_1, 4,\frac43}}  2^{\frac{k_2}2} \| \psi_2 \|_{S^{s_2}_{k_2}} \\
      \ls &  2^{\frac{2k_1+k}2} 2^{\frac{j-k_1}4} 2^{-\frac{k_1+j_1}2} \|   \phi  \|_{S_k^+}   \| Q_{ j_1}^{s_1}  \psi_1 \|_{X^{s_1, \frac12}_{k_1,4,\frac43}}   \| \psi_2 \|_{S^{s_2}_{k_2}}. 
\end{split}
  \]
  From this it follows that
  \[
  |  I_{11}(\mathbf{k})| \ls \sum_{-k_1 \preceq j \prec k_1} | I^j_{11}(\mathbf{k})| \ls 2^{\frac{2k_1+k}2} 2^{- \frac{k_1+j_1}2} \|   \phi  \|_{S_k^+}   \| Q_{ j_1}^{s_1}  \psi_1 \|_{X^{s_1,\frac12}_{k_1,4,\frac43}}   \| \psi_2 \|_{S^{s_2}_{k_2}}. 
  \]
    The case $j \succeq k_1$ is similar, but easier. There is no angular separation and we use  the estimate 
    \[
     \| Q_{ j_1}^{s_1}  \psi_1\|_{L^4_{t,x}} \ls 2^{k_1} 2^{-\frac{j_1}2} \| Q_{ j_1}^{s_1}  \psi_1 \|_{X^{s_1, \frac12}_{k_1,4,\frac43}}, 
  \]
which is obtained in a similar way as above.  From this we obtain
   \[
   |  I^j_{12}(\mathbf{k})| \ls 2^{\frac{2k_1+k}2}  2^{\frac{k_1-j}2} 2^{- \frac{(k_1+j_1)}2}  \|  Q_j^+ \phi  \|_{S_k^+}   \| Q_{ j_1}^{s_1}  \psi_1 \|_{X^{s_1,\frac12}_{k_1,4,\frac43}}   \| \psi_2 \|_{S^{s_2}_{k_2}}. 
   \]
   and further that
   \[
    |  I_{12}(\mathbf{k})| \ls \sum_{j \succeq k_1} |  I^j_{12}(\mathbf{k})| \ls 
    2^{\frac{2k_1+k}2}  2^{- \frac{(k_1+j_1)}2}  
    \|   \phi  \|_{S_k^+}   \| Q_{ j_1}^{s_1}  \psi_1 \|_{X^{s_1, \frac12}_{k_1,4,\frac43}}   \| \psi_2 \|_{S^{s_2}_{k_2}}. 
   \]
 Thus we conclude with 
\begin{equation} \label{I1kb}
|  I_{1}(\mathbf{k})| \ls 2^{\frac{2k_1+k}2}  2^{-\frac{(k_1+j_1)}2} \| \phi  \|_{S_k^+}   \| Q_{j_1}^{s_1} \psi_1 \|_{X^{s_1, \frac12}_{k_1,4,\frac43}}  \| \psi_2 \|_{S^{s_2}_{k_2}}. 
\end{equation}

{\it Contribution of $I_2(\mathbf{k})$:} This case is entirely similar to $I_1(\mathbf{k})$ , just that we switch the use of norms on $\phi$ and $\psi_2$. We claim the following:
\begin{equation} \label{I2kb}
|  I_{2}(\mathbf{k})| \ls 2^{\frac{2k_1+k}2}  2^{- \frac{(k_1+j_1)}2} \| \phi  \|_{S_k^+}   \| Q_{j_1}^{s_1} \psi_1 \|_{X^{s_1, \frac12}_{k_1,4,\frac43}}  \| \psi_2 \|_{S^{s_2}_{k_2}}, 
\end{equation}
and leave the details as an exercise. 
From \eqref{I0kb},  \eqref{I1kb} and  \eqref{I2kb}, we obtain 
\[
\begin{split}
  \sum_{j_1 \succeq -k_1} |I_{j_1}(\mathbf{k}) | \ls 
  2^{\frac{2k+k_1}2}   \|  \phi\|_{S^+_{k}} \| Q_{\succeq -k_1} ^{s_1} \psi_1\|_{X_{k_1,4,\frac43}^{s_1,\frac12,\infty}} \|  \psi_2 \|_{S^{s_2}_{k_2}},
  \end{split}
  \]
which leads the following bound
\begin{equation} \label{tG2cc2}
\tilde G^2(\mathbf{k}) = 2^{\frac{3k}2}  2^{-k_1(-r+\rl)} 2^{-(r+\frac12)k-r k_2} \approx 2^{k_1(r-\rl)} 2^{k(1-2r)},
\end{equation}
for this case. 

\vspace{.2in}

 {\bf Case 3:} $|k_1-k_2 | \leq 10$.

{\it Contribution of $I_0(\mathbf{k})$:}  We split
  $I_0(\mathbf{k})=I_{01}(\mathbf{k})+I_{02}(\mathbf{k})$ according to
 $j_1 \prec k_1$ and $j_1\succeq k_1$. Again, by Lemma \ref{lem:mod} there is
  no contribution if $j_1 \prec k_1$ in the case $s_1=+,s_2=-$ or $s_1=-,s_2=+$ and $k \prec \min (k_1,k_2)$.  With all other choices of signs, we can restrict the analysis of $I_{01}$ to the case $j_1\succeq -k$; by Lemma \ref{lem:mod} with $2l=k_1+k_2-k-j_1\sim 2k_1-k-j_1$ we could localize in caps of angular size $2^{-l}$, but since the structure for $Q_{j_1}^{s_1} \psi_1$ comes at angular scale $l'=\frac{k_1-j_1}2 \leq l$ it is more convenient to localize at this scale; thus we write
  \begin{align*}
I_{01}(\mathbf{k})=  \sum_{\kappa_1,\kappa_2 \in \mathcal{K}_{l'} \atop
      \dist(s_1 \kappa_1,s_2\kappa_2)\ls 2^{-l'}} 
      \Big\{\int  Q^+_{\leq j_1} \phi \cdot \la P_{\kappa_1}Q_{j_1}^{s_1} \psi_1, \beta  P_{\kappa_2} Q_{\leq j_1}^{s_2}  \psi_2 \ra dx dt\Big\},
  \end{align*}
and estimate as follows:
 \[
  \begin{split}
|  I_{01}(\mathbf{k})| \ls &  \|Q_{\leq j_1}^+ \phi  \|_{L^\frac83_t L^8_x} 
\sum_{\kappa_1,\kappa_2 \in \mathcal{K}_{l'} \atop    \dist(s_1 \kappa_1,s_2\kappa_2)\ls 2^{-l'}} 
   2^{-l'}   \| Q_{ j_1}^{s_1}  P_{\kappa_1}\psi_1\|_{L^4_t L^\frac43_x} \|Q_{\leq j_1}^{s_2} P_{\kappa_2}\psi_2\|_{L^\frac83_t L^8_x } \\
  \ls &   \|Q_{\leq j_1}^+ \phi  \|_{L^\frac83_t L^8_x} 2^{-l'} 2^{-\frac{j_1}2}  \| Q_{ j_1}^{s_1} \psi_1 \|_{X^{s_1, \frac12}_{k_1,4,\frac43}}  \| \psi_2 \|_{L^\frac83_t L^8_x[k_2;l']} \\
  \ls &  2^{\frac{3k}4} \| \phi  \|_{S_k^+} 2^{-\frac{k_1}2}  \| Q_{ j_1}^{s_1} \psi_1 \|_{X^{s_1 ,\frac12}_{k_1,4,\frac43}}  2^{\frac{3k_2}4} \| \psi_2 \|_{S^{s_2}_{k_2}} \\
   \ls &   2^{\frac{3k+k_2}4} \| \phi  \|_{S_k^+}  \| Q_{ j_1}^{s_1} \psi_1 \|_{X^{s_1,\frac12}_{k_1,4,\frac43}}   \| \psi_2 \|_{S^{s_2}_{k_2}}. 
\end{split}
  \]
For the case $j_1 \succeq k_1$ we forgo the angular localization and estimate directly
 \[
  \begin{split}
|  I_{02}(\mathbf{k})| \ls &  \|Q_{\leq j_1}^+ \phi  \|_{L^\frac83_t L^8_x}     \| Q_{ j_1}^{s_1}  \psi_1\|_{L^4_t L^\frac43_x} \|Q_{\leq j_1}^{s_2} \psi_2\|_{L^\frac83_t L^8_x } \\
  \ls &  2^{\frac{3k}4} \| \phi  \|_{S_k^+} 2^{-\frac{j_1}2}  \| Q_{ j_1}^{s_1} \psi_1 \|_{X^{s_1, \frac12}_{4,\frac43}}  2^{\frac{3k_2}4} \| \psi_2 \|_{S^{s_2}_{k_2}} \\
   \ls &   2^{\frac{3k+k_2}4}  2^{\frac{k_1-j_1}2}  \| \phi  \|_{S_k^+}  \| Q_{ j_1}^{s_1} \psi_1 \|_{X^{s_1,\frac12}_{4,\frac43}}   \| \psi_2 \|_{S^{s_2}_{k_2}}. 
\end{split}
 \]
Thus we conclude with
\begin{equation} \label{I0kc}
|  I_{0}(\mathbf{k})| \ls 2^{\frac{3k+k_2}4} \min(1,2^{\frac{k_1-j_1}2}) \| \phi  \|_{S_k^+}  \| Q_{ j_1}^{s_1} \psi_1 \|_{X^{s_1, \frac12}_{k_1,4,\frac43}}   \| \psi_2 \|_{S^{s_2}_{k_2}}. 
\end{equation}

{\it Contribution of $I_1(\mathbf{k})$:}  
We split  $I_1(\mathbf{k})=I_{11}(\mathbf{k})+I_{12}(\mathbf{k})$ according to
 $j \prec k_1$ and $j \succeq k_1$. Again, by Lemma \ref{lem:mod} there is
  no contribution if $j \prec k_1$ in the case $s_1=+,s_2=-$ or $s_1=-,s_2=+$ and $k \prec \min (k_1,k_2)$.  With all other choices of signs, we can restrict the sum in $I_{11}$ to $j\succeq -k$, so that by Lemma \ref{lem:mod} with $2l=2k_1-k-j$ can localize in caps of size $2^{-l}$ in the following sense
 \begin{align*}
I^j_{11}(\mathbf{k})=  \sum_{\kappa_1,\kappa_2 \in \mathcal{K}_l \atop
      \dist(s_1 \kappa_1,s_2\kappa_2)\ls 2^{-l}} 
      \Big\{\int  Q^+_{j} \phi \cdot \la P_{\kappa_1}Q_{j_1}^{s_1} \psi_1, \beta  P_{\kappa_2} Q_{\leq j}^{s_2}  \psi_2 \ra dx dt\Big\}.
  \end{align*}
    The angular scale used for $Q_{j_1}^{s_1} \psi_1$ in $X^{s_1,\frac12}_{4,\frac43}$ is $l'=\frac{k_1-j_1}2$; from
   \eqref{Kllg} we obtain
\[
\left( \sum_{\kappa_1 \in \mathcal{K}_{l}}  \| P_{\kappa_1} Q^{s_1}_{j_1} \psi_1 \|_{L^4_t L^\frac43_x} ^2 \right)^\frac12
\ls 2^{\frac{|k_1+j_1-k-j|}4} 2^{-\frac{j_1}2} \|Q_{j_1}^{s_1} \psi_1\|_{X_{k_1,4,\frac43}^{s_1,\frac12}},
\]
  Based on this we estimate as follows
  \[
  \begin{split}
|  I^j_{11}(\mathbf{k})| \ls &  \|Q_{ j}^+ \phi  \|_{L^2_{t} L^\infty_x} 
\sum_{\kappa_1,\kappa_2 \in \mathcal{K}_{l} \atop    \dist(s_1 \kappa_1,s_2\kappa_2)\ls 2^{-l}} 
   2^{-l}   \| Q_{ j_1}^{s_1}  P_{\kappa_1}\psi_1\|_{L^4_t L^\frac43_x} \|Q_{\leq j}^{s_2} P_{\kappa_2}\psi_2\|_{L^4_{t,x} } \\
  \ls &  2^{k} \|Q_{ j}^+ \phi  \|_{L^2_{t,x}} 2^{-l} 2^{\frac{|k_1+j_1-k-j|}4}  2^{-\frac{j_1}2}  \| \psi_1 \|_{X^{s_1 ,\frac12}_{k_1,4,\frac43}}  \| Q_{\leq j}^{s_2} \psi_2 \|_{L^4_{t,x}[k_2;l]} \\
  \ls &  2^{k-\frac{j}2} \| \phi  \|_{S_k^+}  2^{-\frac{2k_1-k-j}2} 2^{\frac{|k_1+j_1-k-j|}4}  2^{-\frac{j_1}2}  \| \psi_1 \|_{X^{s_1 ,\frac12}_{k_1,4,\frac43}}  2^{\frac{k_2}2} \| \psi_2 \|_{S^{s_2}_{k_2}} \\
   \ls &   2^{\frac{3k-k_1-j_1}2}  2^{\frac{|k_1+j_1-k-j|}4} \| \phi  \|_{S_k^+}  \| \psi_1 \|_{X^{s_1 ,\frac12}_{k_1,4,\frac43}}   \| \psi_2 \|_{S^{s_2}_{k_2}}. 
\end{split}
\]
From this we obtain
\[
|  I_{11}(\mathbf{k})| \ls \sum_{-k \preceq j \prec k_1} |  I^j_{11}(\mathbf{k})| \ls
\max( 2^{\frac{6k-k_1-j_1}4}, 2^{2k-\frac{3(k_1+j_1)}4} )  \| \phi  \|_{S_k^+}  \| \psi_1 \|_{X^{s_1 ,\frac12}_{k_1,4,\frac43}}   \| \psi_2 \|_{S^{s_2}_{k_2}}, 
\]
by considering separately the cases $k_1+j_1 \geq 2k$ and $k_1+j_1 \leq 2k$. 

We continue with the estimate for $I_{12}(\mathbf{k})$, that is for the terms with $j \succeq k_1$. In this case we forgo the use of the null condition in order to fit all the possible choices for $s_1,s_2$ into the same analysis. From \eqref{Kllg} we obtain
\[
 \| Q^{s_1}_{j_1} \psi_1 \|_{L^4_t L^\frac43_x} 
\ls \la 2^{\frac{k_1-j_1}4} \ra 2^{-\frac{j_1}2} \|Q_{j_1}^{s_1} \psi_1\|_{X_{k_1,4,\frac43}^{s_1,\frac12}}.
\]
Then we estimate as follows:
\[
  \begin{split}
|  I^j_{12}(\mathbf{k})| \ls &  \|Q_{ j}^+ \phi  \|_{L^2_{t} L^\infty_x} \| Q_{ j_1}^{s_1} \psi_1\|_{L^4_t L^\frac43_x} \|Q_{\leq j}^{s_2} \psi_2\|_{L^4_{t,x} } \\
  \ls &   2^{k-\frac{j}2} \| \phi  \|_{S_k^+}  \la 2^{\frac{k_1-j_1}4} \ra 2^{-\frac{j_1}2}  \| \psi_1 \|_{X^{s_1, \frac12}_{k_1, 4,\frac43}}  2^{\frac{k_2}2} \| \psi_2 \|_{S^{s_2}_{k_2}} \\
   \ls &  2^{k+k_1} 2^{\frac{k_1-j}2} 2^{-\frac{k_1+j_1}2} \| \phi  \|_{S_k^+}  \| \psi_1 \|_{X^{s_1, \frac12}_{k_1,4,\frac43}}   \| \psi_2 \|_{S^{s_2}_{k_2}}. 
\end{split}
\]
From this we obtain
\[
|  I_{12}(\mathbf{k})| \ls \sum_{j \succeq k_1} |  I^j_{12}(\mathbf{k})| \ls
2^{k+k_1} 2^{-\frac{k_1+j_1}2} \| \phi  \|_{S_k^+}  \| \psi_1 \|_{X^{s_1 ,\frac12}_{k_1,4,\frac43}}   \| \psi_2 \|_{S^{s_2}_{k_2}}, 
\]
Adding up the contributions of $ I_{11}(\mathbf{k})$ and $ I_{12}(\mathbf{k})$, and considering the two cases $k_1+j_1 \leq 2k$ and $k_1+j_1 \geq 2k$ gives the following unifying bound
\begin{equation} \label{I1kc}
  | I_{1}(\mathbf{k})| \ls  2^{k+k_1} 2^{-\frac{k_1+j_1}4} \| \phi  \|_{S_k^+}  \| \psi_1 \|_{X^{s_1, \frac12}_{k_1,4,\frac43}}   \| \psi_2 \|_{S^{s_2}_{k_2}}. 
  \end{equation}

{\it Contribution of $I_2(\mathbf{k})$:} A quick inspection of the argument for $I_1(\mathbf{k})$ shows that we can use the same approach with only two modifications:

\noindent
- we use the $L^4_t L^\infty_x$ norm on $\phi$ via the embedding
\[
\|Q_{< j_2}^+ \phi  \|_{L^4_{t} L^\infty_x}  \ls 2^{\frac{k}2} \|Q_{ j}^+ \phi  \|_{L^4_{t,x}} \ls 2^k \| \phi  \|_{S_k^+};
\]
- we use (localized versions of) the $L^2_t L^4_x$ norm on $\psi_2$
\[
\left(  \sum_{\kappa \in \mathcal{K}_l} \| Q^{s_2}_{j_2} P_{\kappa} \psi_2 \|_{L^2_t L^4_x}^2  \right)^\frac12 
\ls 2^{\frac{k_2}2} \| Q^{s_2}_{j_2} \psi_2 \|_{L^2_{t,x}} \ls 2^{\frac{k_2}2} 2^{-\frac{j_2}2} \|\psi_2\|_{S^{s_2}_{k_2}}. 
\]
Thus we claim 
 \begin{equation} \label{I2kc}
  | I_{2}(\mathbf{k})| \ls  2^{k+k_1} 2^{-\frac{k_1+j_1}4} \| \phi  \|_{S_k^+}  \| \psi_1 \|_{X^{s_1 ,\frac12}_{k_1,4,\frac43}}   \| \psi_2 \|_{S^{s_2}_{k_2}},
  \end{equation}
and leave the details as an exercise.

From \eqref{I0kc},  \eqref{I1kc} and  \eqref{I2kc}, we obtain 
\[
\begin{split}
 \sum_{j_1 \succeq -k_1} |I_{j_1}(\mathbf{k}) | & \ls 
 \sum_{j_1 \succeq -k_1} 2^{k+k_1} 2^{-\frac{k_1+j_1}8}    \|  \phi\|_{S^+_{k}} \| Q_{j_1} ^{s_1} \psi_1\|_{X_{k_1,4,\frac32}^{s_1,\frac12}} \|  \psi_2 \|_{S^{s_2}_{k_2}} \\
& \ls 2^{k+k_1}   \|  \phi\|_{S^+_{k}} \| Q_{\succeq -k_1} ^{s_1} \psi_1\|_{X_{k_1,4,\frac32}^{s_1,\frac12,\infty}} \|  \psi_2 \|_{S^{s_2}_{k_2}}. 
  \end{split}
  \]
From this we obtain the following bound
\begin{equation} \label{tG2cc3}
\tilde G^2(\mathbf{k}) \ls 2^{k+k_1}  2^{-k_1(-r+\rl)} 2^{-(r+\frac12)k-r k_2} \approx 2^{k(\frac12-r)} 2^{k_1(1-\rl)},
\end{equation}
in this case. Reviewing \eqref{tG2cc1}, \eqref{tG2cc2} and \eqref{tG2cc3} shows that if $r \geq \frac12$ and $r_0 \geq 1$, then $\tilde G^2$ satisfies the estimate \eqref{G}.

\end{proof}

\begin{proof}[Proof of \eqref{cunn22}]

We fix $j \succeq -k$ and let 
\[
I_{j}(\mathbf{k}) = \int Q_{j}^{+} \phi \la  \psi_1, \beta  \psi_2 \ra dx dt. 
\]
The estimate in \eqref{cunn222} will be retrieved by summing with respect to $j$ the estimates on $I_{j}(\mathbf{k})$.  We decompose $ I_{j}(\mathbf{k})=I_{0}(\mathbf{k})+I_{1}(\mathbf{k})+I_{2}(\mathbf{k})$,  where
  \begin{align*}
    I_0(\mathbf{k}):=&  \int Q_{j}^+ \phi \, \la Q^{s_1}_{\leq j}\psi_1, \beta Q^{s_2}_{\leq j}\psi_2 \ra dx dt,
    \ \ 
    I_1(\mathbf{k}):=\sum_{j_1 > j} \int Q_{j}^+ \phi  \,\la Q^{s_1}_{j_1} \psi_1, \beta Q^{s_2}_{\leq j_1}\psi_2 \ra dxdt,\\
    I_2(\mathbf{k}):=&\sum_{j_2 > j}\int Q_{j}^+ \phi  \,\la   Q^{s_1}_{< j_2}\psi_1, \beta Q^{s_2}_{j_2}\psi_2 \ra dx dt. \\
  \end{align*}
  Here we abuse notation by not using the factor $j$ in the notation for $ I_0(\mathbf{k}),  I_1(\mathbf{k}),  I_2(\mathbf{k})$; this is done in order to have simpler notation later. We split the argument into three cases.

\medskip

     {\bf Case 1:} $|k-k_2 | \leq 10$.

  {\it Contribution of $I_0(\mathbf{k})$:} We split the analysis of
  $I_0(\mathbf{k})$ according to whether
 $j \prec k_1$ or $j \succeq k_1$. By Lemma \ref{lem:mod} there is
  no contribution if $j_1 \prec k_1$ in the case $s_1=+,s_2=-$.  With all other choices of signs, we can restrict the analysis of $I_{01}$ to $j\succeq -k_1$, so that by Lemma \ref{lem:mod} with $2l=k_1+k_2-k-j \sim k_1-j$  we have
\begin{align*}
I_{0}(\mathbf{k})= \sum_{\kappa_1,\kappa_2 \in \mathcal{K}_l \atop
      \dist(s_1 \kappa_1,s_2\kappa_2)\ls 2^{-l}} 
      \Big\{\int  Q^+_{j} \phi \cdot \la P_{\kappa_1}Q_{\leq j}^{s_1} \psi_1, \beta  P_{\kappa_2} Q_{\leq j}^{s_2}  \psi_2 \ra dx dt\Big\}.
  \end{align*}
  From \eqref{Kllg} we have the following estimate
  \begin{equation} \label{aux2}
  \| Q^+_{j} \phi \|_{L^4_t L^\frac43_x} \ls \max(2^{\frac{k-j}4},1)\cdot 2^{-\frac{j}2} \|Q_{j}^+ \phi  \|_{X^{+,\frac12}_{k,4,\frac43}}.  
  \end{equation}
Using this we estimate as follows
\[
  \begin{split}
  |I_{0}(\mathbf{k})| \ls &  \|Q_{j}^+ \phi  \|_{L^4_t L^\frac43_x} 
  \sum_{\kappa_1,\kappa_2 \in \mathcal{K}_l \atop      \dist(s_1 \kappa_1,s_2\kappa_2)\ls 2^{-l}   } 2^{-l} 
  \| Q_{\leq j}^{s_1}  P_{\kappa_1}\psi_1\|_{L^\frac83_t L^8_x} \|Q_{\leq j}^{s_2} P_{\kappa_2}\psi_2\|_{L^\frac83_t L^8_x } \\
  \ls &  2^{\frac{k-j}4} 2^{-\frac{j}2} 2^{-l} \|Q_{j}^+ \phi  \|_{X^{+, \frac12}_{k,4,\frac43}}  \| \psi_1 \|_{L^\frac83_t L^8_x[k_1;l]}  \| \psi_2 \|_{L^\frac83_t L^8_x[k_2;l]} \\
  \ls &  2^{\frac{k-j}4} 2^{-\frac{k_1}2} \|Q_{j}^+ \phi  \|_{X^{+, \frac12}_{k,4,\frac43}} 2^{\frac{3k_1}4} \| \psi_1 \|_{S^{s_1}_{k_1}}  2^{\frac{3k_2}4} \| \psi_2 \|_{S^{s_2}_{k_2}} \\
  \ls &  2^{-\frac{k+j}4} 2^{\frac{k_1+5k_2}4} \|Q_{j}^+ \phi  \|_{X^{+, \frac12}_{k,4,\frac43}}  \| \psi_1 \|_{S^{s_1}_{k_1}}  \| \psi_2 \|_{S^{s_2}_{k_2}}. 
\end{split}
  \]
 In the range $j \succeq k_1$, a similar argument with $l=0$,  gives the bound
  \begin{equation} \label{I0c1e}
  \begin{split}
  |I_{0}(\mathbf{k})| \ls  & \ \max(2^{\frac{k-j}4},1) \cdot 2^{-\frac{j}2} \|Q_{j}^+ \phi  \|_{X^{+,\frac12}_{k,4,\frac43}} 2^{\frac{3k_1}4} \| \psi_1 \|_{S^{s_1}_{k_1}}  2^{\frac{3k_2}4} \| \psi_2 \|_{S^{s_2}_{k_2}} \\
  \ls & \ 2^{k_2} 2^{\frac{k_1-j}2} \|Q_{j}^+ \phi  \|_{X^{+,\frac12}_{k,4,\frac43}} \| \psi_1 \|_{S^{s_1}_{k_1}}  \| \psi_2 \|_{S^{s_2}_{k_2}}. 
  \end{split}
  \end{equation}

  {\it Contribution of $I_1(\mathbf{k})$:} We split
  $I_1(\mathbf{k})=I_{11}(\mathbf{k})+I_{12}(\mathbf{k})$ according to
 $j_1 \prec k_1$ and $j_1\succeq k_1$. Again, by Lemma \ref{lem:mod} there is
  no contribution if $j_1 \prec k_1$ in the case $s_1=+,s_2=-$.  With all other choices of signs, we can restrict the sum in $I_{11}$ to $j_1\succeq -k_1$, so that by Lemma \ref{lem:mod} with $2l=k_1+k_2-k-j_1\sim k_1-j_1$ we have
\begin{align*}
I_{11}^{j_1}(\mathbf{k})=  \sum_{\kappa_1,\kappa_2 \in \mathcal{K}_l \atop
      \dist(s_1 \kappa_1,s_2\kappa_2)\ls 2^{-l}} 
      \Big\{\int  Q^+_{j} \phi \cdot \la P_{\kappa_1}Q_{j_1}^{s_1} \psi_1, \beta  P_{\kappa_2} Q_{\leq j_1}^{s_2}  \psi_2 \ra dx dt\Big\}.
  \end{align*}
  Using Bernstein we obtain
\[
\|P_{\kappa_1}Q_{j_1}^{s_1} \psi_1\|_{L^2_t L^\infty_x} \ls 2^{\frac{2k_1-l}2} \|P_{\kappa_1}Q_{j_1}^{s_1} \psi_1\|_{L^2_{t,x}}
\]
Using this estimate and \eqref{aux2} allows us to estimate
  \begin{align*}
  |I^{j_1}_{11}(\mathbf{k})| & \ls{} \|Q_{j}^+ \phi  \|_{L^4_t L^\frac43_x} 
   \sum_{\kappa_1,\kappa_2 \in \mathcal{K}_l \atop      \dist(s_1 \kappa_1,s_2\kappa_2)\ls 2^{-l}   } 2^{-l}
  \| Q_{j_1}^{s_1}  P_{\kappa_1}\psi_1\|_{L^2_t L^\infty_x} \|Q_{\leq j_1 }^{s_2} P_{\kappa_2}\psi_2\|_{L^4_{t,x} } \\ \\
  & \ls{}  2^{\frac{k-j}4} 2^{-\frac{j}2} \|Q^+_{j} \phi\|_{X^{+, \frac12}_{k,4,\frac43}}
  2^{-\frac{k_1-j_1}2} 2^{\frac{2k_1-l}2}  \| Q^{s_1}_{j_1} \psi_1\|_{L^2_{t,x}} 
      \|  Q_{\leq j_1}^{s_2}  \psi_2 \|_{L^4_{t,x}[k_2;l]}  \\
& \ls{}  2^{\frac{k-3j}4} \|Q^+_{j} \phi\|_{X^{+, \frac12}_{k,4,\frac43}}  2^{\frac{k_1-l}2}  \|\psi_1\|_{S^{s_1}_{k_1}}  2^{\frac{k_2}{2}} \|\psi_2\|_{S^{s_2}_{k_2}}\\
   & \ls{} 2^{\frac{k_1+3k_2}2} 2^{-\frac{3(k+j)}4} 2^{\frac{j_1-k_1}4} \|Q^+_{j} \phi\|_{X^{+,\frac12}_{k,4,\frac43}} \|\psi_1\|_{S^{s_1}_{k_1}} 
   \| \psi_2 \|_{S^{s_2}_{k_2}}.
  \end{align*}
From this we obtain
\[
|I_{11}(\mathbf{k})| \ls \sum_{-k_1 \preceq j_1 \prec k_1} |I^{j_1}_{11}(\mathbf{k})| \ls{} 2^{\frac{k_1+3k_2}2} 2^{-\frac{3(k+j)}4} \|Q^+_{j} \phi\|_{X^{+,\frac12}_{k,4,\frac43}} \|\psi_1\|_{S^{s_1}_{k_1}} 
   \| \psi_2 \|_{S^{s_2}_{k_2}}.
\]
In the range $j_1\succeq k_1$ a similar argument with $l=0$,  gives the bound
\begin{align*}
  |I^{j_1}_{12}(\mathbf{k})| & \ls{} \|Q_{j}^+ \phi  \|_{L^4_t L^\frac43_x}  
  \| Q_{j_1}^{s_1}  \psi_1\|_{L^2_t L^\infty_x} \| Q_{\leq j_1 }^{s_2} \psi_2\|_{L^4_{t,x} } \\ \\
  & \ls{}  \max(2^{\frac{k-j}4},1) 2^{-\frac{j}2} \|Q^+_{j} \phi\|_{X^{+, \frac12}_{4,\frac43}} 2^{k_1}  \| Q^{s_1}_{j_1} \psi_1\|_{L^2_{t,x}} 
    2^{\frac{k_2}{2}} \|\psi_2\|_{S^{s_2}_{k_2}}    \\
& \ls{}  2^{\frac{k-3j}4} \max(1,2^{-\frac{k-j}4}) \|Q^+_{j} \phi\|_{X^{+, \frac12}_{4,\frac43}}  2^{\frac{2k_1-j_1}2}  \|\psi_1\|_{S^{s_1}_{k_1}}  2^{\frac{k_2}{2}} \|\psi_2\|_{S^{s_2}_{k_2}}\\
   & \ls{} 2^{\frac{k_1+3k_2}2} 2^{-\frac{3(k+j)}4} 2^{\frac{k_1-j_1}2} \max(1,2^{-\frac{k-j}4}) \|Q^+_{j} \phi\|_{X^{+,\frac12}_{k,4,\frac43}} \|\psi_1\|_{S^{s_1}_{k_1}} 
   \| \psi_2 \|_{S^{s_2}_{k_2}}.
  \end{align*}
From this we obtain
\[
|I_{12}(\mathbf{k})| \ls \sum_{j_1 \succeq k_1} |I^{j_1}_{12}(\mathbf{k})| \ls{} 2^{\frac{k_1+3k_2}2} 2^{-\frac{3(k+j)}4} \max(1,2^{-\frac{k-j}4}) \|Q^+_{j} \phi\|_{X^{+,\frac12}_{k,4,\frac43}} \|\psi_1\|_{S^{s_1}_{k_1}} 
   \| \psi_2 \|_{S^{s_2}_{k_2}}.
\]
From the two estimates above on $I_{11}(\mathbf{k})$ and $I_{12}(\mathbf{k})$ we conclude with 
\begin{equation} \label{I1c1e}
|I_{1}(\mathbf{k})|  \ls{} 2^{\frac{k_1+3k_2}2} 2^{-\frac{3(k+j)}4} \max(1,2^{-\frac{k-j}4}) \|Q^+_{j} \phi\|_{X^{+,\frac12}_{k,4,\frac43}} \|\psi_1\|_{S^{s_1}_{k_1}} 
   \| \psi_2 \|_{S^{s_2}_{k_2}}.
\end{equation}
 {\it Contribution of $I_2(\mathbf{k})$:}
 We split
  $I_2(\mathbf{k})=I_{21}(\mathbf{k})+I_{22}(\mathbf{k})$ according to
 $j_2 \prec k_1$ and $j_2\succeq k_1$. Again, by Lemma \ref{lem:mod} there is
  no contribution if $j_2 \prec k_1$ in the case $s_1=+,s_2=-$.  With all other choices of signs, we can restrict the sum in $I_{21}$ to $j_2\succeq -k_1$, so that by Lemma \ref{lem:mod} with $2l=k_1+k_2-k-j_2 \sim k_1-j_2$ we have
\begin{align*}
I_{21}^{j_2}(\mathbf{k})=  \sum_{\kappa_1,\kappa_2 \in \mathcal{K}_l \atop
      \dist(s_1 \kappa_1,s_2\kappa_2)\ls 2^{-l}} 
      \Big\{\int  Q^+_{j} \phi \cdot \la P_{\kappa_1}Q_{< j_2}^{s_1} \psi_1, \beta  P_{\kappa_2} Q_{j_2}^{s_2}  \psi_2 \ra dx dt\Big\}.
  \end{align*}
  Using Bernstein we obtain
\[
\|P_{\kappa_1}Q_{< j_2}^{s_1} \psi_1\|_{L^4_t L^\infty_x} \ls 2^{\frac{2k_1-l}4} \|P_{\kappa_1}Q_{j_1}^{s_1} \psi_1\|_{L^4_{t,x}}, \quad 
\|P_{\kappa_2}Q_{j_2}^{s_2} \psi_2\|_{L^2_t L^4_x} \ls 2^{\frac{2k_2-l}4} \|P_{\kappa_1}Q_{j_2}^{s_2} \psi_1\|_{L^2_{t,x}}.
\]
Using these estimates and \eqref{aux2} allows us to estimate
  \begin{align*}
  |I^{j_2}_{21}(\mathbf{k})| & \ls{} \|Q_{j}^+ \phi  \|_{L^4_t L^\frac43_x} 
   \sum_{\kappa_1,\kappa_2 \in \mathcal{K}_l \atop      \dist(s_1 \kappa_1,s_2\kappa_2)\ls 2^{-l}   } 2^{-l}
  \| Q_{<j_2}^{s_1}  P_{\kappa_1}\psi_1\|_{L^4_t L^\infty_x} \|Q_{ j_2 }^{s_2} P_{\kappa_2}\psi_2\|_{L^2_t L^4_{x} } \\ \\
  & \ls{}  2^{\frac{k-j}4} 2^{-\frac{j}2} \|Q^+_{j} \phi\|_{X^{+, \frac12}_{k,4,\frac43}}
  2^{-\frac{k_1-j_2}2}  2^{\frac{2k_1-l}4} \| Q^{s_1}_{< j_2} \psi_1\|_{L^4_{t,x}[k_1;l]}  2^{\frac{2k_2-l}4} \|  Q_{j_2}^{s_2}  \psi_2 \|_{L^2_{t,x}}  \\
& \ls{}  2^{\frac{k-3j}4} \|Q^+_{j} \phi\|_{X^{+, \frac12}_{k,4,\frac43}}  
2^{-\frac{k_1-j_2}2} 2^{\frac{k_1+k_2-l}2}  2^{\frac{k_1}2} \|\psi_1\|_{S^{s_1}_{k_1}}  2^{-\frac{j_2}2} \|\psi_2\|_{S^{s_2}_{k_2}}\\
   & \ls{} 2^{\frac{k_1+3k_2}2} 2^{-\frac{3(k+j)}4} 2^{\frac{j_2-k_1}4} \|Q^+_{j} \phi\|_{X^{+,\frac12}_{k,4,\frac43}} \|\psi_1\|_{S^{s_1}_{k_1}} 
   \| \psi_2 \|_{S^{s_2}_{k_2}}.
  \end{align*}
From this we obtain
\[
|I_{21}(\mathbf{k})| \ls \sum_{-k_1 \preceq j_2 \prec k_1} |I^{j_2}_{21}(\mathbf{k})| \ls{} 2^{\frac{k_1+3k_2}2} 2^{-\frac{3(k+j)}4} \|Q^+_{j} \phi\|_{X^{+,\frac12}_{k,4,\frac43}} \|\psi_1\|_{S^{s_1}_{k_1}} 
   \| \psi_2 \|_{S^{s_2}_{k_2}}.
\]
In the range $j_1\succeq k_1$ a similar argument with $l=0$
  gives the bound
\begin{align*}
  |I^{j_1}_{12}(\mathbf{k})| & \ls{} \|Q_{j}^+ \phi  \|_{L^4_t L^\frac43_x}  
  \| Q_{<j_2}^{s_1}  \psi_1\|_{L^4_t L^\infty_x} \| Q_{j_2 }^{s_2} \psi_2\|_{L^2_t L^4_{x} } \\ \\
  & \ls{}  \max(2^{\frac{k-j}4},1) 2^{-\frac{j}2} \|Q^+_{j} \phi\|_{X^{+, \frac12}_{k,4,\frac43}} 2^{\frac{k_1}2}  \| Q^{s_1}_{< j_2} \psi_1\|_{L^4_{t,x}} 
    2^{\frac{k_2}{2}} \| Q_{j_2 }^{s_2}\psi_2\|_{L^2_{t,x}}    \\
& \ls{}  2^{\frac{k-3j}4} \max(1, 2^{-\frac{k-j}4}) \|Q^+_{j} \phi\|_{X^{+, \frac12}_{k,4,\frac43}}  2^{k_1}  \|\psi_1\|_{S^{s_1}_{k_1}}  2^{\frac{k_2-j_2}{2}} \|\psi_2\|_{S^{s_2}_{k_2}}\\
   & \ls{} 2^{\frac{k_1+3k_2}2} 2^{-\frac{3(k+j)}4} \max(1, 2^{-\frac{k-j}4})2^{\frac{k_1-j_2}2} \|Q^+_{j} \phi\|_{X^{+,\frac12}_{k,4,\frac43}} \|\psi_1\|_{S^{s_1}_{k_1}} 
   \| \psi_2 \|_{S^{s_2}_{k_2}}.
  \end{align*}
From this we obtain
\[
|I_{22}(\mathbf{k})| \ls \sum_{j_2 \succeq k_1} |I^{j_2}_{22}(\mathbf{k})| \ls{} 2^{\frac{k_1+3k_2}2} 2^{-\frac{3(k+j)}4} \max(1, 2^{-\frac{k-j}4}) \|Q^+_{j} \phi\|_{X^{+,\frac12}_{k,4,\frac43}} \|\psi_1\|_{S^{s_1}_{k_1}} 
   \| \psi_2 \|_{S^{s_2}_{k_2}}.
\]
From the two estimates above on $I_{11}(\mathbf{k})$ and $I_{12}(\mathbf{k})$ we conclude with 
\begin{equation} \label{I2c1e}
|I_{2}(\mathbf{k})|  \ls{} 2^{\frac{k_1+3k_2}2} 2^{-\frac{3(k+j)}4} \max(1, 2^{-\frac{k-j}4}) \|Q^+_{j} \phi\|_{X^{+,\frac12}_{k,4,\frac43}} \|\psi_1\|_{S^{s_1}_{k_1}} 
   \| \psi_2 \|_{S^{s_2}_{k_2}}.
\end{equation}
Adding the estimates \eqref{I0c1e}, \eqref{I1c1e} and \eqref{I2c1e}, and performing the summation with respect to $j \succeq -k$ gives the estimate
\[
\Big| \int Q_{\succeq -k}^+ \phi \la \psi_1, \beta  \psi_2 \ra dx dt\Big| \ls{}  2^{\frac{k_1+3k_2}2}
     \| Q_{\succeq -k}^+ \phi\|_{ X^{+, \frac12, \infty}_{k,4,\frac43}}\|\psi_1\|_{S^{s_1}_{k_1}}\|\psi_2\|_{S^{s_2}_{k_2}},
\]
thus we obtain
\begin{equation} \label{tG1cc1}
\tilde G^1(\mathbf{k})= 2^{\frac{k_1+3k_2}2} 2^{-rk_1} 2^{-(\rl+\frac12)k}= 2^{(\frac12-r)k_1} 2^{(1-\rl)k},
\end{equation}
in this case. 

\medskip

     {\bf Case 2:} $|k-k_1 | \leq 10$. This is identical to Case 1.

\medskip

     {\bf Case 3:} $|k_1-k_2 | \leq 10$.

      {\it Contribution of $I_0(\mathbf{k})$:} By Lemma \ref{lem:mod} there is
  no contribution if $j \prec k_1$ in the case $s_1=+,s_2=-$ or $s_1=-,s_2=+$ and $k \prec \min (k_1,k_2)$.  In all remaining cases, we can restrict the analysis to $j\succeq -k$. Assume $j \prec k$; by Lemma \ref{lem:mod} with $2l=k_1+k_2-k-j \sim 2k_1-k-j$  we have
\begin{equation} \label{aux3}
\begin{split}
I_{0}(\mathbf{k})= \sum_{\kappa_1,\kappa_2 \in \mathcal{K}_l \atop
      \dist(s_1 \kappa_1,s_2\kappa_2)\ls 2^{-l}} 
      \Big\{\int  Q^+_{j} \phi \cdot \la P_{\kappa_1}Q_{\leq j}^{s_1} \psi_1, \beta  P_{\kappa_2} Q_{\leq j}^{s_2}  \psi_2 \ra dx dt\Big\}.
  \end{split}
  \end{equation}
 Using \eqref{aux2} we estimate as follows
\[
  \begin{split}
  |I_{0}(\mathbf{k})| \ls &  \|Q_{j}^+ \phi  \|_{L^4_t L^\frac43_x} 
   \sum_{\kappa_1,\kappa_2 \in \mathcal{K}_l \atop      \dist(s_1 \kappa_1,s_2\kappa_2)\ls 2^{-l}   } 2^{-l}
  \| Q_{\leq j}^{s_1}  P_{\kappa_1}\psi_1\|_{L^\frac83_t L^8_x} \|Q_{\leq j}^{s_2} P_{\kappa_2}\psi_2\|_{L^\frac83_t L^8_x } \\
  \ls &  2^{\frac{k-j}4} 2^{-\frac{j}2} 2^{-l} \|Q_{j}^+ \phi  \|_{X^{+, \frac12}_{4,\frac43}}  \| Q_{\leq j}^{s_1} \psi_1 \|_{L^\frac83_t L^8_x[k_1;l]}  \| Q_{\leq j}^{s_2} \psi_2 \|_{L^\frac83_t L^8_x[k_2;l]} \\
  \ls &  2^{\frac{k-j}4} 2^{\frac{k-2k_1}2} \|Q_{j}^+ \phi  \|_{X^{+, \frac12}_{4,\frac43}} 2^{\frac{3k_1}4} \| \psi_1 \|_{S^{s_1}_{k_1}}  2^{\frac{3k_2}4} \| \psi_2 \|_{S^{s_2}_{k_2}} \\
  \ls &  2^{-\frac{k+j}4} 2^{\frac{2k+k_1}2} \|Q_{j}^+ \phi  \|_{X^{+, \frac12}_{4,\frac43}}  \| \psi_1 \|_{S^{s_1}_{k_1}}  \| \psi_2 \|_{S^{s_2}_{k_2}}. 
\end{split}
  \]
We now consider the range $j \succeq k$. In the case $s_1=s_2$ by Lemma \ref{lem:mod} we have that \eqref{aux3} holds true  with $l= k_1-k$; in the case $s_1=-,s_2=+$ and $k \succeq \min(k_1,k_2)$ we choose $l=0$. Then the same argument used above gives the estimate
\[
  \begin{split}
  |I_{0}(\mathbf{k})| \ls 2^{\frac{k+k_1}2} 2^{\frac{k-j}2} \|Q_{j}^+ \phi  \|_{X^{+, \frac12}_{k,4,\frac43}}  \| \psi_1 \|_{S^{s_1}_{k_1}}  \| \psi_2 \|_{S^{s_2}_{k_2}}. 
\end{split}
  \]
  In the case $s_1=+, s_2=-$ or $s_1=-, s_2=+$ and $k \succeq \min(k_1,k_2)$, Lemma \ref{lem:mod} implies there is a contribution only if $j \succeq k_1$. Then a similar argument with $l=0$ gives the estimate:
    \[
  |I_{0}(\mathbf{k})| \ls  \ 2^{k_1} 2^{\frac{k_1-j}2} \|Q_{j}^+ \phi  \|_{X^{+,\frac12}_{k,4,\frac43}} \| \psi_1 \|_{S^{s_1}_{k_1}}  \| \psi_2 \|_{S^{s_2}_{k_2}}. 
  \] 
Thus a common estimate for all cases is
\begin{equation} \label{I0c3e}
|I_{0}(\mathbf{k})| \ls 2^{\frac{k+2k_1}2} (2^{-\frac{k+j}4}+ \mathbf{1}_{j \succeq k_1} 2^{\frac{k_1-j}2}) \|Q_{j}^+ \phi  \|_{X^{+,\frac12}_{k,4,\frac43}} \| \psi_1 \|_{S^{s_1}_{k_1}}  \| \psi_2 \|_{S^{s_2}_{k_2}}, 
\end{equation}
where in the above we have used the notation $\mathbf{1}_{j \succeq k_1}$ for the characteristic function of the set $\{j \in \Z: j \succeq k_1\}$.

{\it Contribution of $I_1(\mathbf{k})$:} 
By Lemma \ref{lem:mod} there is
  no contribution if $j_1 \prec k_1$ in the case $s_1=+,s_2=-$ or $s_1=-,s_2=+$ and $k \prec \min (k_1,k_2)$.  In all remaining cases, we can restrict the analysis to $j_1 \succeq -k$. Assume $j_1 \prec k$; by Lemma \ref{lem:mod} with $2l=k_1+k_2-k-j_1 \sim 2k_1-k-j_1$  we have
\begin{equation}
\begin{split}
I_{1}^{j_1}(\mathbf{k})= \sum_{\kappa_1,\kappa_2 \in \mathcal{K}_l \atop
      \dist(s_1 \kappa_1,s_2\kappa_2)\ls 2^{-l}} 
      \Big\{\int  Q^+_{j} \phi \cdot \la P_{\kappa_1}Q_{j_1}^{s_1} \psi_1, \beta  P_{\kappa_2} Q_{\leq j_1}^{s_2}  \psi_2 \ra dx dt\Big\}.
  \end{split}
  \end{equation}
 Using \eqref{aux2} and Bernstein gives the estimate
 \[
\| Q^+_{j} \phi \|_{L^4_{t,x}} \ls 2^k \| Q^+_{j} \phi \|_{L^4_t L^\frac43_x} \ls 2^{k-\frac{j}2} \max(2^{\frac{k-j}4},1) \|Q_{j}^+ \phi  \|_{X^{+,\frac12}_{k,4,\frac43}} \ls
2^{\frac{3k-j}2}  \|Q_{j}^+ \phi  \|_{X^{+,\frac12}_{k,4,\frac43}}.
 \]
 From this we obtain
\[
  \begin{split}
  |I_{1}^{j_1}(\mathbf{k})| \ls &  \|Q_{j}^+ \phi  \|_{L^4_{t,x}} 
   \sum_{\kappa_1,\kappa_2 \in \mathcal{K}_l \atop      \dist(s_1 \kappa_1,s_2\kappa_2)\ls 2^{-l}   } 2^{-l}
  \| Q_{j_1}^{s_1}  P_{\kappa_1}\psi_1 \|_{L^2_{t,x}} \|Q_{\leq j_1}^{s_2} P_{\kappa_2}\psi_2\|_{L^4_{t,x} } \\
  \ls &  2^{\frac{3k-j}2} 2^{-l} \|Q_{j}^+ \phi  \|_{X^{+, \frac12}_{k,4,\frac43}}  \| Q_{j_1}^{s_1} \psi_1 \|_{L^2_{t,x}}  \| Q_{\leq j_1}^{s_2} \psi_2 \|_{L^4_{t,x}[k_2;l]} \\
  \ls &  2^{\frac{3k-j}2} 2^{-l} \|Q_{j}^+ \phi  \|_{X^{+, \frac12}_{k,4,\frac43}} 2^{-\frac{j_1}2} \| \psi_1 \|_{S^{s_1}_{k_1}}  2^{\frac{k_2}2} \| \psi_2 \|_{S^{s_2}_{k_2}} \\
  \ls &  2^{\frac{5k-k_1}2}  2^{-\frac{k+j}2} \|Q_{j}^+ \phi  \|_{X^{+, \frac12}_{k,4,\frac43}}  \| \psi_1 \|_{S^{s_1}_{k_1}}  \| \psi_2 \|_{S^{s_2}_{k_2}}. 
\end{split}
  \]
  We now consider the range $j_1 \succeq k$. In the case $s_1=s_2$ by Lemma \ref{lem:mod} we have that \eqref{aux3} holds true  with $l= k_1-k$; in the case $s_1=-,s_2=+$ and $k \succeq \min(k_1,k_2)$ we choose $l=0$. Then the same argument used above gives the estimate
 \[
  \begin{split}
  |I_{1}^{j_1}(\mathbf{k})| \ls 2^{\frac{5k-k_1}2}  2^{-\frac{k+j}2} 2^{\frac{k-j_1}2}  \|Q_{j}^+ \phi  \|_{X^{+, \frac12}_{k,4,\frac43}}  \| \psi_1 \|_{S^{s_1}_{k_1}}  \| \psi_2 \|_{S^{s_2}_{k_2}}. 
\end{split}
  \]
 In the case $s_1=+, s_2=-$ or $s_1=-, s_2=+$ and $k \succeq \min(k_1,k_2)$, Lemma \ref{lem:mod} implies there is a contribution only if $j_1 \succeq k_1$. Then a similar argument with $l=0$ gives the estimate:
    \[
  \begin{split}
  |I_{1}^{j_1}(\mathbf{k})| \ls  & \ 2^{\frac{3k-j}2}  \|Q_{j}^+ \phi  \|_{X^{+,\frac12}_{k,4,\frac43}} 2^{-\frac{j_1}2} \| \psi_1 \|_{S^{s_1}_{k_1}}  2^{\frac{k_2}2} \| \psi_2 \|_{S^{s_2}_{k_2}} \\
  \ls & \ 2^{2k}  2^{-\frac{k+j}2} 2^{\frac{k_1-j_1}2} \|Q_{j}^+ \phi  \|_{X^{+,\frac12}_{k,4,\frac43}} \| \psi_1 \|_{S^{s_1}_{k_1}}  \| \psi_2 \|_{S^{s_2}_{k_2}}. 
  \end{split}
  \]
Summing with respect to $j_1$ gives us the following common estimate
\begin{equation} \label{I1c3e}
    |I_{1}(\mathbf{k})| \ls \la k \ra 2^{2k} 2^{-\frac{k+j}2} \|Q_{j}^+ \phi  \|_{X^{+,\frac12}_{k,4,\frac43}} \| \psi_1 \|_{S^{s_1}_{k_1}}  \| \psi_2 \|_{S^{s_2}_{k_2}}. 
\end{equation}
We note that $\la k \ra$ appears only from the summation in the regime $j_1 \prec k$; should we have worked only in the high frequency regime $j \succeq k$, the above scenario would have been unnecessary and we could avoid the logarithmic loss. 

It is easy to see that the contribution of $I_2(\mathbf{k})$
has similar bounds to the ones for $I_1(\mathbf{k})$.

Adding the estimates \eqref{I0c3e} and  \eqref{I1c3e} (these apply to $I_{2}(\mathbf{k})$), and performing the summation with respect to $j \succeq -k$ gives the estimate
\[
\Big| \int Q_{\succeq -k}^+ \phi \la \psi_1, \beta  \psi_2 \ra dx dt\Big| \ls{}  \la k \ra 2^{k+k_1}
     \| Q_{\succeq -k}^+ \phi\|_{ X^{+, \frac12, \infty}_{k,4,\frac43}}\|\psi_1\|_{S^{s_1}_{k_1}}\|\psi_2\|_{S^{s_2}_{k_2}},
\]
thus we obtain
\begin{equation} \label{tG1cc3}
    \tilde G^1(\mathbf{k})= 2^{k+k_1}  2^{(r-\rl-\frac12)k} 2^{-2rk_1}= 2^{(1-2r)k_1} 2^{(r+\frac12-\rl)k}.
\end{equation}
Reviewing \eqref{tG1cc1} and \eqref{tG1cc3}, it follows that $\tilde G^1$ satisfies  \eqref{G} under the assumption that $r \geq \frac12$ and $r_0 \geq 1$. 

\end{proof}

\subsection{Proof of Theorem \ref{thm:main}} \label{subsect:proof} This argument is entirely similar to the one in \cite{BH-DKG}, except for the use of different resolution spaces. 

Again, for notational convenience, let $m=M=1$. Fix $\eps>0$. We will construct a solution \[(\psi_+,\psi_-,\phi_+)\in \mathbf{Z}^{\frac12+\eps}:=Z^{+,\frac12+\eps}\times Z^{-,\frac12+\eps}\times Z^{+,1+\eps}\] of the system \eqref{DKGf} in integral form, i.e.\
\begin{align*}
\psi_+(t)=&e^{- it \la D \ra}\Pi_+(D)\psi_0 +i\int_0^t e^{-i(t-s) \la D \ra}\Pi_+(D)[\Re\phi_+\beta (\psi_++\psi_-)]ds\\
\psi_-(t)=&e^{it \la D \ra}\Pi_-(D)\psi_0 +i\int_0^t e^{i(t-s) \la D \ra}\Pi_-(D)[\Re\phi_+\beta (\psi_++\psi_-)]ds\\
\phi_+(t)=&e^{- it \la D \ra}\phi_{+,0} +i\int_0^t e^{- i(t-s) \la D \ra}\la D\ra^{-1}\la (\psi_++\psi_-),\beta (\psi_++\psi_-)\ra ds,
\end{align*}
provided that the initial data satisfy $\|\psi_{0}\|_{H^{\frac12+\eps}(\R^2)}\leq \delta,\|\phi_{+,0}\|_{H^{1+\eps}(\R^2)}\leq \delta$, for sufficiently small $\delta>0$. We note that in the definition of $Z^\pm_{k}$ in \eqref{Zdef}, and hence in that of $Z^{\pm,\sigma}$, we did not specify the value of $r_0$. Here we set the value $r_0=1$.

Let $T(\psi_+,\psi_-,\phi_+)$ denote the operator defined by the right hand side of the above formula. Using Proposition \ref{mainpro} and Lemma \ref{lem:lin}, we conclude with
\begin{align*}
\|T(\psi_+,\psi_-,\phi_+)\|_{\mathbf{Z}^{\frac12+\eps}}
\ls{} &\delta +\|\phi_+\|_{Z^{+,1+\eps}}(\|\psi_+\|_{Z^{+,\frac12+\eps}}+\|\psi_-\|_{Z^{-,\frac12+\eps}}) +(\|\psi_+\|_{Z^{+,\frac12+\eps}}+\|\psi_-\|_{Z^{-,\frac12+\eps}})^2\\
\ls{} &\delta+\|(\psi_+,\psi_-,\phi_+)\|_{\mathbf{Z}^{\frac12+\eps}}^2,
\end{align*}
and similar estimates for differences. Hence, in a small closed ball in the complete space $\mathbf{Z}^{\frac12+\eps}$ we can invoke the contraction mapping principle to obtain a unique solution. Further, continuous dependence on the initial data is an easy consequence.

It remains to prove that these solutions scatter, which we will only do for $t \to +\infty$, the other case being similar.
It suffices to show that for a solution $(\psi_+,\psi_-,\phi_+)\in \mathbf{Z}^{\frac12+ \eps}$ we have convergence of the integrals that appear above in the expressions of $\psi_\pm(t)$ and $\phi_+(t)$. 

This is a byproduct of the linear theory provided by Lemma \ref{lem:lin}.
Indeed, by Remark \ref{rmk:u2} it follows that on the dyadic level these integrals are in fact in $U^{\pm}_k$ and this is square-summable. From this it follows that they are in the space
\[ V^2(\R;H^{\frac12+\eps}(\R^2)) \times  V^2(\R;H^{\frac12+\eps}(\R^2))
\times  V^2(\R;H^{1+\eps}(\R^2)). 
\]
Functions of bounded $2-$variation have limits at infinity \cite[Prop.~2.2]{HHK} which proves the scattering claim.

\bibliographystyle{plain} \bibliography{dkg-refs}

\end{document}